\documentclass[12pt]{amsart}
\usepackage[pdfusetitle,colorlinks,citecolor=blue,urlcolor=black,linkcolor=black]{hyperref}
\usepackage{enumitem,stmaryrd,amssymb}

\theoremstyle{plain}
\newtheorem{thm}{Theorem}[section]
\newtheorem{claim}{Claim}[thm]
\newtheorem{lemma}[thm]{Lemma}
\newtheorem{fact}[thm]{Fact}

\newtheorem{cor}[thm]{Corollary}
\newtheorem*{thma}{Theorem}

\theoremstyle{definition}
\newtheorem*{definition}{Definition}
\newtheorem{defn}[thm]{Definition}
\newtheorem{notation}[thm]{Notation}
\newtheorem{conv}[thm]{Convention}

\theoremstyle{remark}
\newtheorem{remark}[thm]{Remark}

\DeclareMathOperator{\ord}{Ord}

\DeclareMathOperator{\otp}{otp}
\DeclareMathOperator{\cf}{cf}
\DeclareMathOperator{\dom}{dom}

\DeclareMathOperator{\acc}{acc}
\DeclareMathOperator{\nacc}{nacc}

\newcommand*\axiomfont[1]{\textsf{\textup{#1}}}
\newcommand\zfc{\axiomfont{ZFC}}

\DeclareMathOperator{\onto}{\textsf{onto}}
\DeclareMathOperator{\bd}{bd}
\DeclareMathOperator{\hull}{Hull}

\newcommand{\one}{\mathop{1\hskip-3pt {\rm l}}}
\newcommand\res{\mathbin{\upharpoonright}}
\newcommand\rest{\mathbin{\upharpoonright\mkern-12.5mu\raisebox{0.7pt}{$\upharpoonright$}}}
\newcommand\rests{\mathbin{\upharpoonright\mkern-9mu\raisebox{0.1pt}{$\upharpoonright$}}}

\newcommand{\forces}{\Vdash}
\newcommand{\nforces}{\nVdash}

\newcommand\br{\blacktriangleright}
\newcommand\s{\subseteq}
\renewcommand\mid{\mathrel{|}\allowbreak}

\title{Failure of an higher analogue of mho}
\author{Ido Feldman}
\address{Department of Mathematics, Bar-Ilan University, Ramat-Gan 5290002, Israel.}
\email{ido.feldman@live.com}

\begin{document}
\begin{abstract} Justin Moore's weak club-guessing principle $\mho$
admits various possible generalizations to the second uncountable cardinal.
One of them was shown to hold in $\zfc$ by Shelah.
A stronger one was shown to follow from several consequences of the continuum hypothesis by Inamdar and Rinot.
Here we prove that the stronger one may consistently fail. Specifically, starting with a supercompact cardinal and an inaccessible cardinal above it,
we devise a notion of forcing consisting of finite working parts and finitely many two types of models as side conditions,
to violate this analog of $\mho$ at the second uncountable cardinal.
\end{abstract}
\maketitle
\section{Introduction}
In \cite{MR2199228}, Moore proved it is consistent for the class of uncountable linear orders to admit a $5$-element basis.
In particular, this shows that the class of Aronszajn lines admits a $2$-element basis, solving a longstanding question of Shelah from \cite{AbSh:114}.
Later on, in \cite{MR2444284}, Moore proved a complementary result providing a very weak sufficient condition for the class of Aronszajn lines to not admit a basis of size two.
Moore's principle is denoted $\mho$ and is a very weak club-guessing principle,
which cannot be destroyed by an $\omega$-proper notion of forcing (see \cite{MR3821630}).
In the notation of the following natural generalization, Moore's principle $\mho$ coincides with $\mho(\omega_1,\omega)$.
\begin{defn} For a stationary subset $S$ of a regular uncountable cardinal $\kappa$,
and for a cardinal $\theta<\kappa$, $\mho(S,\theta)$ asserts the existence of a sequence $\langle (h_\delta,C_\delta)\mid \delta\in S\rangle$
such that:
\begin{itemize}
\item For every $\delta\in S$, $h_\delta$ is a function from $\delta$ to $\theta$;
\item For every $\delta\in S$, $C_\delta$ is a club in $\delta$ of order-type $\cf(\delta)$,
and for every $\alpha<\delta$, $h_\delta(\alpha)=h_\delta(\min(C_\delta\setminus\alpha))$;
\item For every cofinal $A\s\kappa$, there is a $\delta\in S$ such that for every $\tau<\theta$,
$$\sup\{ \alpha\in A\cap\delta\mid \min(C_\delta\setminus\alpha)\in S\ \&\ h_\delta(\alpha)=\tau\}=\delta.$$
\end{itemize}
\end{defn}

Moving one cardinal up, we focus on the stationary set $S^2_1$ of all ordinals below $\omega_2$ of cofinality $\omega_1$.
By \cite[Claim 3.1]{Sh:413}, $\mho(S^2_1,\omega)$ holds outright in $\zfc$,
hence the main open problem is whether this could be improved to get $\mho(S^2_1,\omega_1)$ in $\zfc$.
In \cite[\S4]{paper46}, a sufficient condition for $\mho(S^2_1,\omega_1)$ to hold was given, namely,
the principle $\onto(J^{\bd}[\omega_1],\omega_1)$ from \cite{paper47}
which follows from various weak consequences of the continuum hypothesis.
In \cite[\S6]{MR2298476}, Larson demonstrated that $\onto(J^{\bd}[\omega_1],\allowbreak\omega_1)$ may fail via proper forcing using finitely many countable models as side conditions.
In this paper, Larson's result is extended and the consistency of the failure of $\mho(S^2_1,\omega_1)$ is established.
\begin{thma} Assuming the consistency of a supercompact cardinal and an inaccessible cardinal above it,
it is consistent that $\mho(S^2_1,\omega_1)$ fails.
\end{thma}

The model witnessing the above theorem is obtained by forcing with two types of virtual models as side conditions,
building on the machinery developed in a series of papers by Velickovic and his co-authors.
Precisely, assuming $\kappa$ is a supercompact cardinal and $\lambda$ is an inaccessible cardinal above it,
we iteratively add a tower of $\lambda$ many clubs in $\kappa$, eventually collapsing $\kappa$ to $\aleph_2$,
$\lambda$ to $\aleph_3$, and ensuring that any potential candidate for a witness to $\mho(S^2_1,\omega_1)$
is destroyed by one of the clubs in our tower.
A condition in an `iterand' of our forcing is a triple of finite sets $\langle \mathcal M_p,d_p,F_p\rangle$, where the pair $\langle \mathcal M_p,d_p\rangle$
is a condition in the Velickovic-Mohammadpour poset from \cite{MohammadpourVelickovic}. Shelah's theorem on $\mho(S^2_1,\omega)$ imposes constraints which influences our very definition of $F_p$
and its interaction with the Velickovic-Mohammadpour pair $\langle \mathcal M_p,d_p\rangle$.

\subsection{Organization of this paper}
In Section~\ref{sec2}, we provide some preliminaries on virtual models and the posets $\mathbb M^\kappa_\lambda$ and $\mathbb P^\kappa_\lambda$ from \cite{MohammadpourVelickovic},
where $\kappa<\lambda$ is a pair of inaccessible cardinals.

In Section~\ref{sec3}, we present out forcing notion in two steps. First we define preconditions, and then define actual conditions.

In Section~\ref{sec4}, we verify that our notion of forcing has properness features sufficient to ensure that $\omega_1$ is preserved and $\kappa$ becomes $\omega_2$.

In Section~\ref{sec5}, we show that our forcing has the $\lambda$-cc, hence $\lambda$ becomes $\omega_3$ in the extension.

In Section~\ref{sec6}, we verify that $\mho(S^2_1,\omega_1)$ fails in the extension.

\section{A brief overview on virtual models} \label{sec2}

Throughout, $\kappa<\lambda$ stand for a pair of inaccessible cardinals.
All the content of this section is due to Velickovic and Mohammadpour \cite{MohammadpourVelickovic}
\begin{defn}
A structure $\mathcal A:=(A,\in,\kappa)$ is called \emph{suitable} if $A$ is transitive set and $\mathcal A$ satisfies $\zfc$ and $\kappa$ is translated to be inaccessible in $A$.
\end{defn}

Hereafter, we fix a suitable $\mathcal A:=(A,\in,\kappa)$.
\begin{defn}
Given $\alpha\in A\cap \ord$, let $\mathcal A_\alpha$ be the substructure of $\mathcal A$ with underlying set $A_{\alpha}:=A\cap V_{\alpha}$. Finally let $E:=\{\alpha\in A\mid \mathcal A_{\alpha}\prec\mathcal A\}$.

Given a set $\mathcal M$ of substructures of $\mathcal A$ we denote:
\begin{itemize}
\item[$\br$] $\pi_0(\mathcal M):=\{\mathsf{M}\in\mathcal M\mid \mathsf{M}\text{ is countable}\}$.

\item[$\br$] $\pi_1(\mathcal M):=\{\mathsf{M}\in\mathcal M\mid \mathsf{M}\text{ is uncountable}\}$.

\end{itemize}

In addition, given a substructure $\mathsf{M}$ of $\mathcal A$ we will often concern ourselves with $\delta_\mathsf{M}:=\sup(\mathsf{M}\cap \kappa)$.
Thus, we also denote:
\begin{itemize}
\item[$\br$] $\Delta_0(\mathcal M):=\{\delta_\mathsf{M} \mid \mathsf{M}\in\pi_0(\mathcal M)\}$.

\item[$\br$] $\Delta_1(\mathcal M):=\{\delta_\mathsf{M}\mid \mathsf{M}\in\pi_1(\mathcal M)\}$.
\end{itemize}
\end{defn}

\begin{defn}
For a substructure $\mathsf{M}$ of $\mathcal A=(A,\in,\kappa)$ and subset $X$ of $A$, let
$$\hull(\mathsf{M},X):=\{f(x)\mid f\in \mathsf{M}\text{ is a function, and } x\in\dom(f)\cap [X]^{<\omega}\}.$$
\end{defn}

\begin{fact}
Suppose $\mathsf{M}$ is an elementary submodel of $\mathcal A$ and $X$ is a subset of $A$.
Let $\delta:=\sup(\mathsf{M}\cap\ord)$, and suppose $X\cap A_{\delta}$ is nonempty.
Then, $\hull(\mathsf{M},X)$ is the least elementary submodel of $\mathcal A$ containing $\mathsf{M}$ and $X\cap A_{\delta}$.
\end{fact}

\begin{defn}
Suppose $\alpha\in E$. We let $\mathcal A_{\alpha}$ denote the pool of all transitive $A$'s which are elementary extensions of $V_{\alpha}$ and of the same cardinality.
\end{defn}

\begin{defn}[Virtual models]
Suppose $\alpha\in E$. We let $\mathcal V_{\alpha}$ denote the collection of all substructures $\mathsf{M}$ of $(V_{\lambda},\in,\kappa)$ of size less than $\kappa$, such that if $A:=\hull(\mathsf{M},V_{\alpha})$, then $A\in\mathcal A_{\alpha}$ and $\mathsf{M}$ is elementary in $(A,\in,\kappa)$.

We let $\mathcal V$ denote the union over all $\mathcal V_\alpha$'s.
\end{defn}

We refer to the elements of $\mathcal V_{\alpha}$ as \emph{$\alpha$-models}.
Note that if $\mathsf{M}\in\mathcal V_{\alpha}$, then $\sup(\mathsf{M}\cap\text{ORD})\geq\alpha$.
In general $\mathsf{M}$ is not elementary in $V_{\lambda}$, this happens only if $\mathsf{M}\s V_{\alpha}$,
in which case we think of $\mathsf{M}$ as a \emph{standard} model.
\begin{notation}
For $\mathsf M\in\mathcal V$, we let $\eta(\mathsf M)$ be the unique ordinal such that $\mathsf M\in\mathcal V_{\alpha}$.
\end{notation}
We now have a natural way to define the first type models.
\begin{defn}[Countable virtual models]
For $\alpha\in E$. Let $\mathcal B_{\alpha}$ denote the collection of countable models in $\mathcal V_{\alpha}$. Denote by $\mathcal B_{st}$ the collection of standard models.
\end{defn}

\begin{defn}[Magidor models]
We say a model $\mathsf{M}$ is $\kappa$-Magidor model if for $\bar{\mathsf{M}}$ the transitive collapse of $\mathsf{M}$ and $\pi$ the corresponding collapsing map,
for some $\bar{\gamma}<\kappa$, $\bar{\mathsf{M}}=V_{\bar{\gamma}}$, $\cf(\bar{\gamma})>\pi(\kappa)$ and $V_{\pi(\kappa)}\s \mathsf{M}$.
\end{defn}

\begin{defn}
Let $\mathcal U^{\kappa}_{\alpha}$ denote the collection of all $\mathsf{M}\in\mathcal V_{\alpha}$ that are $\kappa$-Magidor models.

We let $\mathcal U^\kappa$ denote the union over all $\mathcal U^\kappa_\alpha$'s.
\end{defn}
We denote the set of all standard models in $\mathcal U^{\kappa}$ as $\mathcal U^{\kappa}_{\text{st}}$.

\begin{fact}
If $\kappa$ is supercompact, then $\mathcal U_{\text{st}}^{\kappa}$ is stationary in $\mathcal \mathsf{P}_{\kappa}(V_{\lambda})$.
\end{fact}

\begin{defn}
Suppose $\mathsf{M},\mathsf{N}\in\mathcal V$ and $\alpha\in E$.
A map $\sigma:\mathsf{M}\rightarrow \mathsf{N}$ is an \emph{$\alpha$-isomorphism} if it has an extension $\bar{\sigma}:\hull(\mathsf{M},V_{\alpha})\rightarrow \hull(\mathsf{N},V_{\alpha})$ which is isomorphism.
In this case, we write $\mathsf{M}\cong_{\alpha} \mathsf{N}$ and say $\mathsf{M}$ is $\alpha$-isomorphic to $\mathsf{N}$.
\end{defn}

\begin{defn}
Suppose $\alpha,\beta\in E$ and $\mathsf{M}$ is a $\beta$-model. Let $\overline{\hull(\mathsf{M},V_{\alpha})}$ be the transitive collapse of $\hull(\mathsf{M},V_{\alpha})$, and let $\pi$ be the collapsing map. We define $\mathsf{M}\res\alpha$ to be $\pi[\mathsf{M}]$.
\end{defn}
\begin{remark} For $\alpha,\beta\in E$, and $\beta$-models $\mathsf{M},\mathsf{N}$:
\begin{itemize}
\item $\mathsf{M}\res\alpha=\mathsf{N}\res\alpha$ iff $\mathsf{M}\cong_{\alpha} \mathsf{N}$.
\item If $\beta<\alpha$, then $\mathsf{M}=\mathsf{M}\res\alpha$.
\item If $\beta\geq\alpha$, then $\mathsf{M}\res\alpha\cong_{\alpha} \mathsf{M}$.
\item If $\alpha\leq\beta$ , then $(\mathsf{M}\res\beta)\res\alpha=\mathsf{M}\res\alpha$.
\item If $\alpha,\beta\ge\kappa$, then $\delta_{\mathsf{M}\restriction\alpha}=\delta_{\mathsf{M}\restriction\beta}$.
\end{itemize}
\end{remark}

We define a version of the membership relation for $\alpha\in E$.
\begin{defn}
Suppose $\mathsf{M},\mathsf{N}\in\mathcal V$ and $\alpha\in E$. We write $\mathsf{M}\in_{\alpha} \mathsf{N}$ and say that $\mathsf{M}$ is $\alpha$-belongs to $\mathsf{N}$,
if there exists $\mathsf{M}'\cong_{\alpha} \mathsf{M}$ such that $\mathsf{M}'\in \mathsf{N}$.
\end{defn}

\begin{defn}[active points]
Let $\mathsf M\in\mathcal V$ and $\alpha\in E$. We say that $\mathsf M$ is \emph{active at $\alpha$} if $\eta(M)\geq\alpha$ and $H(M, V_{\delta_M})\cap E\cap\alpha$ is unbounded in $E\cap\alpha$.
We say that such $M$ is \emph{strongly active} if $M\cap E\cap\alpha$ is unbounded in $E\cap\alpha$.
\end{defn}

\begin{remark}
For $\mathsf N\in\mathcal U^{\kappa}$ for any $\alpha\in E$, $\mathsf N$ is active at $\alpha$ if and only if $\mathsf N$ is strongly active at $\alpha$.
\end{remark}

\begin{fact}
Given $\mathsf M,\mathsf N\in\mathcal V$ and $\alpha\in E$. If $\mathsf M\cong_{\alpha}\mathsf N$, then $\mathsf M, \mathsf N$ have the same active points up to $\alpha$.
\end{fact}

This fact gives a rise to the following definition.

\begin{defn}
Given $\mathsf M\in\mathcal V$. Let $a(\mathsf M):=\{\alpha\in E\mid\mathsf M\text{ is active at }\alpha\}$.
Let $\alpha(\mathsf M):=\max(a(\mathsf M))$.
\end{defn}

\begin{remark}
Note that $a(M)$ is closed subset of $E$ of size at most $|H(M,V_{\delta_M})|$.
\end{remark}

\begin{defn}[Meet of models]
Suppose $\mathsf{N}\in\mathcal U^\kappa$ and $\mathsf{M}\in\mathcal B$. Let $\alpha:=\max(a(\mathsf{N})\cap a(b))$.
Assuming $\mathsf{N}\in_{\alpha} \mathsf{M}$, we define $\mathsf{N}\wedge \mathsf{M}$ as follows.
First, let $\bar{\mathsf{N}}$ be the transitive collapse of $\mathsf{N}$, and $\pi$ the collapsing map. Set,
$$\eta:=\sup(\sup(\pi^{-1}[\bar{\mathsf{N}}\cap \mathsf{M}]\cap\ord)\cap E\cap(\alpha+1)).$$
Then, the \emph{meet} of $\mathsf{N}$ and $\mathsf{M}$ is $\mathsf{N}\wedge \mathsf{M}:=\pi^{-1}[\bar{\mathsf{N}}\cap \mathsf{M}]\res\eta$
\end{defn}

\begin{lemma}
Let $\mathsf{N}\in\mathcal U^\kappa$ and $\mathsf{M}\in\mathcal B$. Suppose $\alpha\in E$ and the meet $\mathsf{N}\wedge \mathsf{M}$ is defined and is $\alpha$-active.
Then, $(\mathsf{N}\wedge \mathsf{M})\cap V_{\alpha}=\mathsf{N}\cap \mathsf{M}\cap V_{\alpha}$.
\end{lemma}

\begin{lemma}
Let $\alpha\in E$. Given $\mathsf N\in\mathcal U^{\kappa}$, $\mathsf M\in\mathcal B$ and the meet $\mathsf N\wedge\mathsf M$ is defined.
Then,
\begin{enumerate}
\item If both $\mathsf N,\mathsf M$ are strongly active at $\alpha$, then $\mathsf N\wedge\mathsf M$ is strongly active at $\alpha$.
\item $(\mathsf N\wedge\mathsf M)\rest\alpha=(\mathsf N\rest\alpha)\wedge (\mathsf M\rest\alpha)$.
\item For $\mathsf P\in\mathcal V$ active at $\alpha$, $\mathsf P\in_{\alpha}\mathsf N\wedge \mathsf M$ iff $\mathsf P\in_{\alpha}\mathsf N$ and $\mathsf P\in_{\alpha}\mathsf M$.
\end{enumerate}
\end{lemma}

\begin{defn}
The following notation shall be comfortable to use is some cases:
\begin{notation}
Let $\alpha\in E$ and $\mathcal M$ be an $\in_{\alpha}$-chain. Let $\in^{*}_{\alpha}$ be the transitive closure of $\in_{\alpha}$. For $\mathsf M,\mathsf N\in\mathcal M$, we shall say that $\mathsf M$ is $'in_{\alpha}$-below $\mathsf N$ iff $\mathsf M\in^{*}_{\alpha}\mathsf N$.

Moreover, we let:
$$(\mathsf M,\mathsf N)^{\alpha}_{\mathcal M}:=\{\mathsf P\in\mathcal M\mid\mathsf M\in^*_{\alpha}\mathsf P\in^*_{\alpha}\mathsf N\}.$$

We define similarly $[\mathsf M,\mathsf N)^{\alpha}_{\mathcal M}$ and $[\mathsf M,\mathsf N]^{\alpha}_{\mathcal M}$.
\end{notation}

We are now ready to define the very base forcing.

Suppose $\alpha\in E$.
$\mathbb{M}^{\kappa}_\alpha$ stands for the collection of all $\mathcal M$ such that:
\begin{enumerate}
\item $\mathcal M$ is a finite subset of $\mathcal B_{\leq\alpha}\cup \mathcal U_{\leq\alpha}^\kappa$;
\item $\mathcal M^{\delta}:=\{\mathsf{M}\rest\delta\mid \mathsf{M}\in\mathcal M, \delta\in a(\mathsf M)\}$ is an $\in_{\delta}$-chain for all $\delta\in E\cap(\alpha+1)$.

\end{enumerate}
We say that $\mathcal N\leq\mathcal M$ if for all $\mathsf{M}\in\mathcal M$ there exists an $\mathsf{N}\in\mathcal N$ such that $\mathsf{N}\res\eta(\mathsf{M})=\mathsf{M}$. Set $\mathbb{M}^{\kappa}_{\lambda}:=\bigcup\mathbb{M}^{\kappa}_{\alpha}$.
\end{defn}

\begin{notation} We write $\mathcal M(\gamma)$ for $\mathcal M^{\min(E\setminus(\gamma+1))}$.
\end{notation}

Suppose $p\in\mathbb{M}^{\kappa}_{\lambda}$,
$\mathsf{M}\in\pi_0(p)$ a $\beta$-model for some $\beta\in E$.
We would like to define the \emph{residue of $p$ to $\mathsf{M}$},
denoted $p\res\mathsf{M}$, which is an element of $\mathsf{M}$.
For this, we shall scan over all $\delta\in a(\mathsf{M})$,
and all $\mathsf{N}\in_{\delta} \mathsf{M}$ such that $\delta\in a(\mathsf{N})$
in order to identify an $\mathsf{N'}\in\mathsf{M}$ that is $\alpha$-isomorphic to $\mathsf{N}$ for the largest possible $\alpha$.
It will be denoted $\mathsf N\res\mathsf M$. Finally, the residue will be
$$p\res\mathsf M:=\{ \mathsf N\res\mathsf M\mid \mathsf{N}\in p\text{ and }\mathsf N\res\mathsf M\text{ is defined}\}.$$

Let $\mathsf{N}\in\mathcal V_\beta$. If $\mathsf{N}\in_{\delta} \mathsf{M}$ and $\delta\in a(\mathsf{M})\cap a(\mathsf{N})$,
then we shall define $\mathsf N\res\mathsf M$ as follows.
Note that $\mathsf{N}\in_{\alpha} \mathsf{M}$ for $\alpha:=\alpha_{\mathsf{M},\mathsf{N}}$.
Also, note that, if $\mathsf{M}$ is standard $\beta$-model, then $\alpha<\beta$.
Now, consider $\mathsf{X}:=(X,\in,\kappa)$, for $X:=\hull(\mathsf{M},V_{\beta})$.

If $\alpha\in \mathsf{M}$, then we can compute $\mathsf{N}\rest\alpha$ in $\mathsf{M}$,
so we let $\mathsf{N}\res\mathsf{M}:=\mathsf{N}\rest\alpha$.
Otherwise, $\alpha\notin \mathsf{M}$, so we go for the second best option, namely $\alpha^*:=\min((\mathsf{M}\cap\ord)\setminus\alpha)$.
We know $\alpha^*\in E\cap \mathsf{M}$ and is of uncountable cofinality.
As $\mathsf{N}\in_{\alpha} \mathsf{M}$ there is a model $\mathsf{N}'\in \mathsf{M}$ which is $\alpha$-isomorphic to $\mathsf{N}$.
Now, $\mathsf{N}'\res(\alpha^*)\in \mathsf{M}$. Hence we may assume that $\mathsf{N}'\in(\mathcal V_{\alpha^*})^{\mathsf{X}}$.
Moreover, note that $\mathsf{N}'$ is unique.\footnote{Indeed, let $\mathsf{N}''\in \mathsf{M}$ with same property. Since $\alpha^*$ is the least ordinal above $\alpha$ and $\mathsf{N}''\cong_{\alpha} \mathsf{N}'$ we have that $\mathsf{N}''\cong_{\delta} \mathsf{N}'$ for all $\delta\in \mathsf{M}\cap\alpha^*\cap E_A$. Thus, $\mathsf{N}''\cong_{\alpha^*} \mathsf{N}'$. But, $\mathsf{N}'', \mathsf{N}'$ are $\alpha^*$-models. So, $\mathsf{N}''=\mathsf{N}'$. }
In this case, we let $\mathsf{N}\res\mathbf{M}$ be this $\mathsf{N}'$.
We also need to define the residue in the case $\mathsf M\in\pi_1(p)$.
In this case, for any $\mathsf N\in p$, let $\mathsf N\res\mathsf M:=\mathsf N\rest\alpha(\mathsf N,\mathsf M)$ if $\delta_{\mathsf N}<\delta_{\mathsf M}$, otherwise we leave it undefined.
The \emph{residue of $p$ with respect to $\mathsf M$} is set to be:
$$p\res\mathsf M:=\{\mathsf N\res\mathsf M\mid \mathsf{N}\in p\text{ and }\mathsf N\res\mathsf M\text{ is defined}\}.$$

Next the following Fact shall prove itself to be useful for our purposes:
\begin{fact}\label{fact 2.27}
Let $p\in\mathbb M^{\kappa}_{\lambda}$ and $\alpha\in E$.
Given $\mathsf M\in\pi_0(p)$,
for every $\mathsf K\in p^{\alpha}\setminus(p\res\mathsf M)^{\alpha}$,
there exists $\mathsf N\in\pi_1((p\res M)^{\alpha})$ such that, $\mathsf K\in((\mathsf N\wedge\mathsf M)\rest\alpha, \mathsf N)^{\alpha}_{p}$.
\end{fact}

Our next step is adding decorations. In order to establish that, we will need the following definition:
\begin{defn}
Suppose $\mathcal M\in\mathbb M^{\kappa}_{\lambda}$. Let $\mathcal L(\mathcal M):=\{\mathsf M\rest\alpha\mid\mathsf M\in\mathcal M\text{ and }\alpha\in a(\mathsf M)\}$.

In addition, we say that a structure $\mathsf M\in\mathcal L(\mathcal M)$ is \text{$\mathcal M$-free}
if for every $\mathsf N\in\mathcal M$ if $\mathsf M\in_{\eta(\mathsf M)}\mathsf N$, then $\mathsf N$ is strongly active at $\eta(\mathsf M)$.

Denote by $\mathcal F(\mathcal M)$ the set of all $\mathcal M$-free structures.
\end{defn}

\begin{defn}
Suppose $\alpha\in E\cup\{\lambda\}$.
$\mathbb{P}^{\kappa}_{\alpha}$ denotes the collection of all pairs
$p:=(\mathcal M_p, d_p)$, where $\mathcal M_p\in\mathbb{M}^{\kappa}_\alpha$, $d_p$ is a finite partial function from $\mathcal F(\mathcal M_p)$ to $[V_{\kappa}]^{<\omega}$ and
\begin{equation}\tag{$\star$}
\text{if } \mathsf{M}\in \dom(d_p),\, \mathsf{N}\in\mathcal M_p\text{ and } \mathsf{M}\in_{\eta(\mathsf{M})} \mathsf{N},\text{ then }d_p(\mathsf{M})\in \mathsf{N}.
\end{equation}
We say $q\leq p$ iff $\mathcal M_q\leq \mathcal M_p$ and for every $\mathsf{M}\in\dom(d_p)$, there is $\beta\in E\cap(\eta(\mathsf{M})+1)$ such that, $\mathsf{M}\res\beta\in\dom(d_p)$ and $d_p(\mathsf{M}\res\beta)\s d_q(\mathsf{M}\res\beta)$.
\end{defn}

The notion of residue for $\mathbb{M}^{\kappa}_{\alpha}$ extends naturally to $\mathbb{P}^{\kappa}_{\alpha}$, as follows.
\begin{defn}[]\label{d219}
Suppose $p\in\mathbb{P}^{\kappa}_{\lambda}$, $\mathsf{M}\in\pi_0(\mathcal M_p)$ a $\beta$-model for some $\beta\in E$.
Then $p\res \mathsf{M}:=\langle\mathcal M_{p\res \mathsf{M}},d_{p\res \mathsf{M}})$ is defined via:
\begin{enumerate}
\item $\mathcal M_{p\res \mathsf{M}}:=\{\mathsf{N}\res \mathsf{M}\mid \mathsf{N}\in\mathcal M_p\}$.
\item $\dom(d_{p\res \mathsf{M}}):=\{\mathsf{N}\res \mathsf{M}\mid \mathsf{N}\in\dom(d_p)\text{ and } \mathsf{N}\in_{\eta(\mathsf{N})} \mathsf{M}\}$.
\item If $\mathsf{N}\in\dom(d_p)$ and $\mathsf{N}\in_{\eta(\mathsf{N})} \mathsf{M}$, let $d_{p\res \mathsf{M}}(\mathsf{N}\res \mathsf{M})=d_p(\mathsf{N})$.

\end{enumerate}
\end{defn}

And for the second type residue we have the following:
\begin{defn}[]\label{d220}
Suppose $p\in\mathbb{P}^{\kappa}_{\lambda}$, $\mathsf{M}\in\pi_1(\mathcal M_p)$ a $\beta$-model for some $\beta\in E$.
Then $p\res \mathsf{M}:=\langle\mathcal M_{p\res \mathsf{M}},d_{p\res \mathsf{M}})$ is defined via:
\begin{enumerate}
\item $\mathcal M_{p\res \mathsf{M}}:=\{\mathsf{N}\res \mathsf{M}\mid \mathsf{N}\in\mathcal M_p\}$.
\item $d_{p\res\mathsf M}:=d_p\res(\dom(d_p)\cap\mathsf M)$.
\end{enumerate}
\end{defn}

\begin{fact}\label{fact221}
Suppose $p\in\mathbb{P}^{\kappa}_{\lambda}$, $\mathsf{M}\in\mathcal M_p$ a $\beta$-model for some $\beta\in E$.
Then $p\res \mathsf{M}$ is a legitimate condition extended by $p$.
\end{fact}

\begin{conv}
For a forcing notion $\mathbb P$ for two compatible $p,q\in\mathbb P$ we denote $p\wedge q$ to be the weakest condition among all conditions $r$ such that $r\leq p,q$.
\end{conv}

We shall finish this section by recalling that the forcing notion $\mathbb P^{\kappa}_{\lambda}$ is $\mathcal B_{st}$ and $\mathcal U^{\kappa}_{st}$ strongly proper.
\begin{fact}\label{fact234}
For $p\in\mathbb P^{\kappa}_{\lambda}$ suppose $\mathsf M\in\mathcal B_{st}\cup\mathcal U^{\kappa}_{st}$ is such that $\mathsf M\in\mathcal M_p$.
Then, for any condition $q\in\mathbb P^{\kappa}_{\lambda}\cap\mathsf M$ extending $p\res\mathsf M$, $p\wedge q$ exists.
\end{fact}

\section{The forcing poset}\label{sec3}

Hereafter, we assume that $\kappa$ is a supercompact cardinal and that $\lambda$ is an inaccessible cardinal above it.
We first introduce for each $\alpha\in E$ the definition of the set of pre-conditions:
\begin{defn}
For $\alpha\in E$, let $\text{pre}\mathbb O_{\alpha}$ be the set of triples $p:=\langle\mathcal M_p,d_p, F_p\rangle$ where:
\begin{enumerate}
\item[(a)] $\langle\mathcal M_p,d_p\rangle\in\mathbb P^{\kappa}_{\alpha}$;
\item[(b)] $F_p$ is a finite function from $E$ to $H_{\alpha}$ such that:
\begin{enumerate}
\item[(1)] If $F_p\neq\emptyset$, then $\min(E)\in F_p$;
\item[(2)] For every $\gamma\in\dom(F_p)$, $F_p(\gamma)$ is a pair $(B^\gamma_p,\langle\tau^{\gamma}_{\delta,p}\mid\delta\in D^{\gamma}_p\rangle)$.
\item[(3)] For every $\mathsf{M}\in\mathcal M_p$, $\gamma\in\dom(F_p)\cap\mathsf{M}$ and $\delta\in\Delta_1(D^{\gamma}_p)\cap\mathsf{M}$ we have that $\tau^{\gamma}_{\delta,p}\in\mathsf{M}$.

\end{enumerate}
\end{enumerate}

We assert that $p\leq_{\text{pre}\mathbb O_\alpha} q$ if all of the following holds:
\begin{itemize}
\item $\langle\mathcal M_p,d_p\rangle\leq_{\mathbb P^{\kappa}_{\alpha}}\langle\mathcal M_q,d_q\rangle$;
\item $\dom(F_p)\supseteq\dom(F_q)$;
\item For all $\gamma\in\dom(F_q)$, $B^\gamma_p\supseteq B^\gamma_q$ and $D^{\gamma}_p\supseteq D^{\gamma}_q$;
\item For all $\gamma\in\dom(F_q)$ for every $\delta\in D^{\gamma}_q$, $\tau^{\gamma}_{\delta,q}=\tau^{\gamma}_{\delta,p}$.
\end{itemize}

Note that, for $\alpha<\beta$ in $E$ we have that $\text{pre}\mathbb O_{\alpha}\s\text{pre}\mathbb O_{\beta}$.
Moreover, for a given $p\in\text{pre}\mathbb O_{\beta}$ and $\alpha<\beta$ in $\dom(F_p)$ we may define the projection map, mapping $p$ to an element of $\text{pre}\mathbb O_{\alpha}$ via:
$$p\res\alpha:=\langle\mathcal M_p\res\alpha,d^{\alpha}_p,F_p\cap H_{\alpha}\rangle.$$
\end{defn}

We are now ready to define our iteration.
But first we will need the following assumption:

\begin{conv}
From now on, we denote $\xi_0:=\min(E)$.
We also assume all models $\mathsf{M}\in\mathcal B_{st}\cup\mathcal U^{\kappa}_{st}$
have $\xi_0$ as an element.

\end{conv}

Having defined preconditions, we are now ready to define
our `iteration' $\langle \mathbb O_{\alpha}\mid\alpha\in E\rangle$ to kill all possible candidates for $\mho(S^2_1,\omega_1)$.
A candidate for $\mho(S^2_1,\omega_1)$ is a sequence $\langle (h_\delta,C_\delta)\mid \delta\in S^2_1\rangle$ such that for every $\delta\in S^2_1$:
\begin{itemize}
\item $C_\delta$ is a club in $\delta$ of order-type $\omega_1$;
\item $h_\delta$ is a map from $\delta$ to $\omega_1$;
\item for every $\beta<\delta$, $h_\delta(\beta)=h_\delta(\min(C_\delta\setminus\beta))$;
\item for every $\beta\in\nacc(C_\delta)\setminus S^2_1$, $h_\delta(\beta)=0$.
\end{itemize}
Note that we may also assume the following:
\begin{itemize}
\item $\min(C_\delta)=0$ for every $\delta\in S^2_1$;
\item if $\delta\in\acc^+(S^2_1)$, then $\nacc(C_\delta)\s \{0\}\cup S^2_1$.
\end{itemize}
\begin{defn}\label{defn33}
Let $\phi:\lambda\rightarrow H_{\lambda}$ be a surjection such that each element is enumerated cofinally often.
For $\alpha\in E$, $\mathbb O_{\alpha}$ consists of conditions $p:=\langle\mathcal M_p,d_p,F_p\rangle$ such that:
\begin{itemize}
\item[(a)] $p\in\text{pre}\mathbb O_{\alpha}$;
\item[(b)] $F_p$ has the following property:
For $\gamma\in\dom(F_p)$, if $\phi(\gamma)$ is a $\mathbb O_{\gamma}$-name for a candidate $\langle (h^\gamma_{\delta},C^\gamma_\delta)\mid\delta\in S^2_1\rangle$, then we say that $\gamma$ is good and we require:
\begin{itemize}
\item[(1)] $p\res\gamma$ forces that for all $\mathsf{K}\in\pi_1(\mathcal M_p(\gamma))$,
$$\dot{h}_{\delta}^\gamma[\Delta_1(D^{\gamma}_p)]\cap\{\tau^{\gamma}_{\delta_\mathsf{K}}\}=\emptyset.$$
\item[(2)] $\Delta_1(\mathcal M_p(\gamma))$ is a subset of $D^{\gamma}_p$;
\item[(3)] For every $\mathsf{M}\in\pi_0(\mathcal M_p(\xi_0))$ such that $\delta_\mathsf{M}\in D^{\gamma}_p$, for every $\mathsf{N}\in\pi_1(\mathcal M_p(\gamma))$ with $\mathsf{N}\in_{\xi_0} \mathsf{M}$, $\delta_{\mathsf{M}\wedge \mathsf{N}}\in D^{\gamma}_p$.

\item[(4)] For every $\mathsf{K}\in\pi_1(\mathcal M_p(\gamma))$, for every $\mathsf{P} \in\pi_0(\mathcal M_p(\xi_0))$ that is $\in_{\xi_0}$-below $\mathsf{K}$ with $\delta_\mathsf{P}\in D^{\gamma}_p$,
if $\one\nolimits_{\gamma}\nforces\delta_\mathsf{P}\notin\acc^+(\nacc(\dot{C}^{\gamma}_{\delta_\mathsf{K}}\cap\delta_\mathsf{P}))$
and considering
$$ \text{Form}(\mathsf K):=\left\{\mathsf{N}\wedge\mathsf{M} \,\middle|\, \begin{array}{l}
\mathsf{N}\in\pi_1(\mathcal M_p(\gamma))\in_{\xi_0}\text{-above }\mathsf K\\
\mathsf{M}\in\pi_0(\mathcal M_p(\gamma))\text{ contains-}\in_{\xi_0}\mathsf N\\
\delta_M\in B^{\gamma}_p\text{ and }\mathsf{P}=\mathsf{N}\wedge\mathsf{M}.
\end{array} \right\}$$
one of the following holds:
\begin{itemize}
\item[(i)] $\mathsf P\in\text{form}(\mathsf K)$;
\item[(ii)] $\mathsf P\neq(\mathsf K\rest\xi_0)\wedge\mathsf S$ for any $\mathsf S\in\pi_0(\mathcal M_p(\xi_0))$;
\item[(iii)] $\delta_{\mathsf P}\in B^{\gamma}_p$,
\end{itemize}
then
$$p\res\gamma\forces\tau^{\gamma}_{\delta_\mathsf{K}}\notin\dot{h}^{\gamma}_{\delta_\mathsf{K}}[\delta_{\mathsf{P}}].$$
\end{itemize}
\end{itemize}
\end{defn}

We assert that $p\le_{\mathbb O_\alpha} q$ iff $p\leq_{\text{pre}\mathbb O_\alpha} q$.

\begin{conv}
Given $\mathsf{K}\in\pi_1(\mathcal M_p(\gamma))$, $\mathsf{P} \in\pi_0(\mathcal M_p(\xi_0))$ that is $\in_{\xi_0}$-below $\mathsf{K}$ with $\delta_\mathsf{P}\in D^{\gamma}_p$ we say that $\mathsf P$ is of the \emph{right form} for $\mathsf K$ in $p$ if for $\text{form}(\mathsf K)$ one of Clauses (i)--(ii) holds.

Note that if $q\leq p$ and $\mathsf P$ was of the right form for $\mathsf K$ in $p$ then it is of the right form for $\mathsf K$ in $q$.
\end{conv}

\begin{remark}\label{proper remark}
Since we are going to prove that our forcing is proper, given
a function $\dot{h}:\delta\rightarrow\omega_1$ and $\beta<\delta$,
for any condition $p$ which forces $\dot{h}(\beta)\in\omega_1$,
by considering an elementary model $\mathsf{M}$ containing all the relevant objects,
we may extend to a master-condition $p_\mathsf{M}\leq p$.

Now, for any $r\leq p$, there exists $q\in \mathsf{M}$ which is compatible with $r$.
As $\mathsf{M}$ contains $\dot{h},\beta,\delta$, there exists $\tau\in \mathsf{M}$ such that, $q\forces h_{\delta}(\beta)=\tau$.

Thus, there are only countable many different colors our forcing could assign to $\dot{h}_{\delta}(\beta)$.
\end{remark}

We verify that $\mathbb O_\alpha$ is a complete suborder of $\mathbb O_\beta$
whenever $\alpha\le\beta$.

\begin{lemma}\label{complete suborder}
Suppose $\alpha\leq \beta$ in $E$. Let $p\in\mathbb O_{\beta}$ and $q\in\mathbb O_{\alpha}$ such that $q\leq p\res\alpha$. Then, there exists $r\in\mathbb O_{\beta}$ such that $r\leq p,q$.
\end{lemma}
\begin{proof}
We define $r:=\langle \mathcal{M}_r,d_r,F_r\rangle$ as follows:
\begin{itemize}
\item $\mathcal M_r:=\mathcal M_p\cup\mathcal M_q$;
\item $d_r(\mathsf{M}):=\begin{cases} d_p(\mathsf{M}) & \mathsf{M}\in\mathcal M_p;\\
d_q(\mathsf{M}) & \text{otherwise}.
\end{cases} $
\item $\dom(F_r):=\dom(F_p)\cup\dom(F_q)$.
\item $F_r(\gamma):=\begin{cases}
F_p(\gamma),& \gamma\in\dom(F_p)\setminus\alpha;\\
F_q(\gamma),& \gamma\in\dom(F_q).
\end{cases}$
\end{itemize}

The fact that $\langle\mathcal M_r,d_r\rangle\leq\langle\mathcal M_p,d_p\rangle ,\langle\mathcal M_q,d_q\rangle$ is demonstrated in \cite[Lemma 3.8]{MohammadpourVelickovic}.
It is also clear that $r$ satisfies conditions (b1)--(b3) and $r\leq p,q$
More specifically, the only non-trivial part is to show that Clause (b4) holds.
Let $\gamma\in\dom(F_r)$, $\mathsf{K}\in\pi_1(\mathcal M_r(\gamma))$, $\mathsf{P} \in\pi_0(\mathcal M_r(\xi_0))$
that is $\in_{\xi_0}$-below $\mathsf{K}$ with $\delta_\mathsf{P}\in D^{\gamma}_r$.
Moreover, suppose $\one\nolimits_{\gamma}\nforces\delta_\mathsf{P}\notin\acc^+(\nacc(\dot{C}^{\gamma}_{\delta_\mathsf{K}}\cap\delta_\mathsf{P}))$
and $\mathsf P$ is of the right form.

Suppose that $\gamma\in\dom(F_q)$. Then, $\mathcal M_r(\gamma)=\mathcal M_q(\gamma)$ and $D^{\gamma}_r= D^{\gamma}_q$ so $\delta_{\mathsf P}\in D^{\gamma}_q$.
Hence (b4) holds by recalling $q$ is a condition.
Otherwise $\gamma\in\dom(F_p)$.
If $\gamma\geq\alpha$ as $\mathcal M_r(\gamma)=\mathcal M_p(\gamma)$, again by the fact $D^{\gamma}_r=D^{\gamma}_p$ we get that (b4) holds.
Else $\gamma<\alpha$.
As $q\leq p\res\alpha$ the conclusion holds as (b4) holds for $q$.
\end{proof}

The idea is that at stage $\gamma\in E$ we add a generic club $D_{\gamma}$ relative to $S^2_1$
which will witness the failure of the candidate $\langle (h^{\gamma}_{\delta},C^\gamma_\delta)\mid\delta\in (S^2_1)^{G_{\gamma}}\rangle$
given by $\phi(\gamma)$.
Namely, for each $\delta\in D_{\gamma}$ we will have that:
$$h^\gamma_{\delta}[D_\gamma\cap S^2_1]\neq\omega_1.$$

This is witnessed by the color $\tau^{\gamma}_{\delta}$ assigned to $\delta$ by the forcing.

Therefore, one must first verify that our poset preserves $\omega_1,\kappa,\lambda$,
which are collapsed in the final extension to be $\omega_1$, $\omega_2$ and $\omega_3$ respectively.
In order to accomplish our task,
We need to prove that for every $\alpha\in E$, the forcing $\mathbb O_{\alpha}$ is proper and $\kappa$-proper and that $\mathbb O_\lambda$ has the $\lambda$-cc.
This is the subject matter of the next section.

\section{The properness} \label{sec4}

We shall prove all of the properness statements by induction on $\alpha\in E$.
The base case, namely $\alpha=\xi_0$, follows from the fact that $\mathbb O_{\xi_0}$ is isomorphic to $\mathbb P^{\kappa}_{\xi_0}$ and therefore it is even strongly $\mathcal B_{st}$ and $\mathcal U^{\kappa}_{st}$ proper.
Thus, hereafter assume that we are given $\alpha\in E$ above $\xi_0$ satisfying that $\mathbb O_{\gamma}$ is proper for all $\gamma<\alpha$.

\subsection{The residue}

We start by defining the residues of conditions.
\begin{defn}[residue]
Given $\mathsf{M}\in\mathcal B_{st}\cup\mathcal U^{\kappa}_{st}$ and a condition $p$ with $\mathsf{M}\in\mathcal M_p$, \emph{the residue} of $p$ with respect to $\mathsf{M}$
is the triple $p\res\mathsf M:=\langle\mathcal M_{p\res M}, d_{p\res\mathsf{M}},F_{p\res\mathsf{M}}\rangle$
such that:
\begin{enumerate}
\item $\langle\mathcal M_{p\res\mathsf{M}},d_{p\res\mathsf{M}}\rangle:=\langle\mathcal M_p,d_p\rangle\res\mathsf{M}$;
\item $\dom(F_{p\res\mathsf{M}})=\dom(F_p)\cap\mathsf{M}$;
\item For every $\gamma\in\dom(F_{p\res\mathsf{M}})$, set:
\begin{itemize}
\item $B^{\gamma}_{p\res\mathsf M}:=\{\delta\in B^{\gamma}_p\mid\exists\mathsf P\in\mathcal M_{p\res\mathsf M}(\delta=\delta_{\mathsf P})\}$,
\item $D^{\gamma}_{p\res\mathsf M}:=\{\delta\in D^{\gamma}_p\mid\exists\mathsf P\in\mathcal M_{p\res\mathsf M}(\delta=\delta_{\mathsf P})\}$, and
\item $F_{p\res\mathsf{M}}(\gamma):=(B^{\gamma}_{p\res\mathsf M},\langle \tau^{\gamma}_{\delta, p}\mid\delta\in D^{\gamma}_{p\res\mathsf M}\rangle)$.
\end{itemize}
\end{enumerate}

\end{defn}

The following lemma verifies that our definition of a residue is indeed sound.
\begin{lemma}
Given $\mathsf{M}\in\mathcal B_{st}\cup\mathcal U^{\kappa}_{st}$ and a condition $p$ with $\mathsf{M}\in\mathcal M_p$,
$p\res\mathsf{M}$ is a legitimate condition extended by $p$.

\end{lemma}
\begin{proof}
We first verify that $p\res\mathsf{M}$ is indeed a condition. Thus, we have to prove Clauses (a)--(b) of Definition~\ref{defn33} hold:

\begin{itemize}
\item[(a)] This follows directly from Fact~\ref{fact221}.
\item[(b)] We have to verify Clauses (1)--(4):
\begin{enumerate}
\item Suppose there exists $q\leq p\res\mathsf{M}$ such that, for some $\gamma\in\dom(F_{p\res\mathsf{M}})$ and $\delta\in\Delta_1(D^{\gamma}_{p\res\mathsf{M}})$, $q\forces\dot{h}^{\gamma}_{\delta}[\Delta_1(D^{\gamma}_{p\res\mathsf{M}})]\cap\{\tau^{\gamma}_{\delta,p\res\mathsf{M}}\}\neq\emptyset$.
But, as $q\leq p\res\mathsf{M}$ by the definition of extension we get:
\begin{itemize}
\item[$\br$] $D^{\gamma}_{p\res\mathsf{M}}\s D^{\gamma}_{q}$;

\item[$\br$] $\tau^{\gamma}_{\delta,q}=\tau^{\gamma}_{\delta,p}$.
\end{itemize}
So as $q\forces\dot{h}^{\gamma}_{\delta}[\Delta_1(D^{\gamma}_{q})]\cap\{\tau^{\gamma}_{\delta,q}\}=\emptyset$ by Clause (b1).
This is a contradiction.
\item Immediate from the construction.
\item Let $\delta_{\mathsf{P}}\in\Delta_0(D^{\gamma}_{p\res\mathsf{M}})$. Then, for any $\mathsf{N}\in_{\xi_0}\mathsf{P}$, as $\mathsf P\in_{\xi_0}\mathsf M$ and $\mathsf{P}$ is countable, $\mathsf{N}\in_{\xi_0}\mathsf M$ and therefore $\mathsf N\wedge\mathsf P\in_{\xi_0}\mathsf M$.
\item Given $\gamma\in\dom(F_{p\res\mathsf M})$, $\mathsf{K}\in\pi_1(\mathcal M_p(\gamma))$, $\mathsf{P} \in\mathcal M_p(\xi_0)$ countable $\in_{\xi_0}$-below $\mathsf{K}$ with $\delta_\mathsf{P}\in D^{\gamma}_p$.

Suppose $\mathsf{P}$ is of the right form for $\mathsf K$ in $p\res\mathsf M$ or $\delta_{\mathsf P}\in B^{\gamma}_{p\res\mathsf M}$.
Moreover we may assume,
$\one\nolimits_{\gamma}\nforces\delta_\mathsf{P}\notin\acc^+(\nacc(\dot{C}^{\gamma}_{\delta_\mathsf{K}}\cap\delta_\mathsf{P}))$. We need to show that,
$$(p\res\mathsf M)\res\gamma\forces\tau^{\gamma}_{\delta_\mathsf{K}}\notin\dot{h}^{\gamma}_{\delta_\mathsf{K}}[\delta_{\mathsf{P}}].$$
Let $q\leq (p\res\mathsf M)$. Since $\one\nolimits_{\gamma}\nforces\delta_\mathsf{P}\notin\acc^+(\nacc(\dot{C}^{\gamma}_{\delta_\mathsf{K}}\cap\delta_\mathsf{P}))$, $\mathsf P$ is of the right form for $\mathsf K$ in $q$ or, in the other case,
$\delta_{\mathsf P}\in B^{\gamma}_q$ and
$\tau^{\gamma}_{\mathsf K, p\res\mathsf M}=\tau^{\gamma}_{\mathsf K, q}$ we have that,
$$q\res\gamma\forces \tau^{\gamma}_{\delta_\mathsf{K}}\notin\dot{h}^{\gamma}_{\delta_\mathsf{K}}[\delta_{\mathsf{P}}]$$
As desired.
\end{enumerate}
\end{itemize}

In addition, it is clear that $p\leq p\res\mathsf M$\end{proof}

\subsection{Appropriate conditions}
As our main goal in this section is to prove properness, we have to use the following fact from \cite{MR3821634} which enables the amalgamation procedure in the proof that $\mathbb M^{\kappa}_{\lambda}$ is strongly proper.
\begin{fact}\label{Timer fact}
Let $\mathcal M_p\in\mathbb{M}^\kappa_{\alpha}$ and $\mathsf{M}\in\pi_0(\mathcal M_p)$. If $\mathcal M_q\in\mathbb{M}^\kappa_{\alpha}\cap \mathsf{M}$,
then for every $\gamma\in E\cap(\alpha+1)$ every element $\mathsf{P}\in(\mathcal M_p\wedge\mathcal M_q)(\gamma)\setminus(\mathcal M_p(\gamma)\cup\mathcal M_q(\gamma))$ does not belong to $\bigcup\{[\mathsf{N}\wedge \mathsf{M},\mathsf{N})_{p\wedge q}\mid \mathsf{N}\in\mathcal M_{p\rests \mathsf{M}}(\gamma)\}$.
\end{fact}

Unfortunately in our scenario, Fact~\ref{Timer fact} alone is not enough in order to prove properness.
We need to introduce a new demand on a condition $q$ which we call \emph{appropriate conditions} for $p$, in order to amalgamate it with a candidate for master condition $p$.
In particular, since we shall choose a condition in a particular way, the forcing will not be strongly-proper.

This subsection shall describe a relationship between two conditions $p$ and $q$ that would allow us to amalgamate them.
Lemma~\ref{appropriate lemma} captures its main utilization for the amalgamation process.
In order to define the relation, we use the following definition introduced in \cite{rodriguez2022forcingsymmetricsystemsmodels}:

\begin{defn}[Gallart Rodriguez]
Let $\alpha\in E$. For $p\in\mathbb O_{\alpha}$, suppose $\mathsf{M}\in\pi_0(\mathcal M_p)\cap\mathcal B_{st}$. For each $\gamma\in\dom(F_p)$,
we define the set of \emph{$\varepsilon$-heights} $\varepsilon^{\gamma}_{p\res \mathsf{M}}:=\{\varepsilon^{\gamma}_{i},\bar{\varepsilon}^{\gamma}_i\mid i\leq m^{\gamma}_p\}$ as follows:
\begin{itemize}
\item $\varepsilon^{\gamma}_0=\bar{\varepsilon}^{\gamma}_0=\delta_\mathsf{M}$.
\item $\varepsilon^{\gamma}_1:=\max(\Delta_1(\mathcal M_p(\gamma))\cap\delta_\mathsf{M})$;
\item $\bar{\varepsilon}^{\gamma}_1=\delta_{\mathsf{N}\wedge \mathsf{M}}$ for $\mathsf{N}\in\mathcal M_p(\gamma)$ such that $\delta_\mathsf{N}=\varepsilon^{\gamma}_1$;
\item $\varepsilon^{\gamma}_{i+1}=\max(\Delta_1(\mathcal M_p)\cap\bar\varepsilon^\gamma_i)$;
\item $\bar{\varepsilon}^{\gamma}_{i+1}=\delta_{\mathsf{N}\wedge \mathsf{M}}$ for $\mathsf{N}\in\mathcal M_p(\gamma)$ such that $\delta_\mathsf{N}=\varepsilon_{i+1}$.
\end{itemize}
\end{defn}

Now, in addition we will need the following definition:

\begin{defn}
Let $\alpha\in E$. For $p\in\mathbb O_{\alpha}$,
suppose $\mathsf{M}\in\pi_0(\mathcal M_p)\cap\mathcal B_{st}$.
For each $\gamma\in\dom(F_p)\cap \mathsf{M}\cup\{\xi_0\}$,
we define the set of \emph{$\Lambda$-heights} $\Lambda^{\gamma}:=\{\Lambda^{\gamma}_i\mid i\leq m^\gamma_p\}$ as follows:

As a first step, denote $\bar{\Lambda}^{\gamma}_i$ as the first ordinal such that for all $\mathsf{K}\in\Delta^1(\mathcal M_p(\gamma))$,
\begin{enumerate}
\item If $\one\nolimits_{\gamma}$ forces $C^{\gamma}_{\delta_\mathsf{K}}\cap\bar{\varepsilon}^{\gamma}_i$ is bounded, then it forces
$C^{\gamma}_{\delta_\mathsf{K}}\cap[\bar{\Lambda}^{\gamma}_i,\bar{\varepsilon}^{\gamma})$ is empty.
\item If $S$ is countable model in $\mathcal M_p(\gamma)$ with height above $\bar{\varepsilon}^{\gamma}_{i}$ such that $\sup(\bar{\varepsilon}^{\gamma}_i\cap S)<\bar{\varepsilon}^{\gamma}_i$,
then, $\bar{\Lambda}^{\gamma}_i>\sup(\bar{\varepsilon}^{\gamma}_i\cap S)$.

\end{enumerate}

Next, define $\Lambda^{\gamma}_i$ to be the first ordinal above $\bar{\Lambda}^{\gamma}_i$ lying in $(\mathsf{N}_i\wedge \mathsf{M})\cap V_{\kappa}$, for any choice of $\mathsf{N}_i\in\mathcal M_{p\res \mathsf{M}}$ with $\varepsilon^{\gamma}_i=\delta_{\mathsf{N}_i}$, such that
for every countable $S\in\mathcal M_{p}(\gamma)$ with $\sup(\bar{\varepsilon}^{\gamma}_i\cap S)=\bar{\varepsilon}^{\gamma}_i$, it is the case that $\Lambda^{\gamma}_i>\min(S\setminus\bar{\Lambda}^{\gamma}_i)$.

For $\xi_0$ we define $\Lambda^{\xi_0}_i$ the same way except we consider $\mathsf{K}\in\pi_0(\mathcal M_p(\gamma))$ and $S\in\pi_0(\mathcal M_p(\gamma))$ for all $\gamma\in\dom(F_p)\cap\mathsf{M}$.
\end{defn}

Note that for all $i$:
$$\varepsilon^\gamma_{i+1}<\bar\Lambda^\gamma_i<\Lambda^\gamma_i<\bar\varepsilon^\gamma_{i}<\varepsilon^\gamma_{i}.$$

\begin{defn}
Given $\alpha\in E$, for $p\in\mathbb O_{\alpha}$, suppose $\mathsf{M}\in\pi_0(\mathcal M_p)\cap\mathcal B_{st}$,
$q\in\mathbb O_{\alpha}\cap \mathsf{M}$ extending $p\res\mathsf{M}$.
We say that $q$ is \emph{appropriate for $p$} if for every $\gamma\in\dom(F_p)\cap\mathsf{M}$ for every $i<m^{\gamma}_p$,
for every $\mathsf{K}\in\mathcal M_p(\gamma)$ and $\mathsf{P}\in\mathcal M_q(\gamma')\setminus \mathcal M_{p\restriction \mathsf{M}}(\gamma')$ where $\gamma'\in\{\gamma,\xi_0\}$ it is forced that if
$\delta_\mathsf{P}<\bar{\varepsilon}^{\gamma}_i<\delta_\mathsf{K}$,
then
$\delta_{\mathsf{P}}>\Lambda^{\gamma}_i$.
\end{defn}

\begin{lemma}\label{appropriate lemma}

Given $\alpha\in E$, for $p\in\mathbb O_{\alpha}$, suppose $\mathsf{\mathsf{M}}\in\pi_0(\mathcal M_p)\cap\mathcal B_{st}$ with $\xi_0\in\mathsf{M}$
and let $q\in\mathbb O_{\alpha}\cap \mathsf{M}$ extending $p\res\mathsf{\mathsf{M}}$ such that $q$ is appropriate for $p$.
Then, for all $\gamma\in\dom(F_p)\cap \mathsf{M}\cup\{\xi_0\}$ for every $i\leq m^{\gamma}_p$,
for every $\mathsf{K}\in\mathcal M_p(\gamma)$ and $\mathsf{P}\in\mathcal M_q(\gamma)\setminus \mathcal M_{p\res \mathsf{M}}(\gamma)$ $\in_{\gamma}$-below $\mathsf{K}$, such that $\delta_\mathsf{P}<\bar{\varepsilon}^{\gamma}_i<\delta_\mathsf{K}$ then, if $\one\nolimits_{\gamma}\forces\bar{\varepsilon^{\gamma}_i}\notin\acc^+(\nacc(C^{\gamma}_{\delta_\mathsf{K}}))$, then $\one\nolimits_{\gamma}\forces\delta_\mathsf{P}\notin\acc^+(\nacc(C^{\gamma}_{\delta_\mathsf{K}}))$.

\end{lemma}
\begin{proof}
Let $\gamma\in\dom(F_p)\cap \mathsf{M}\cup\{\xi_0\}$, $i\leq m^{\gamma}_p$.
For $\mathsf{K}\in\mathcal M_p(\gamma)$ and $\mathsf{P}\in\mathcal M_q(\gamma)\setminus \mathcal M_{p\res \mathsf{M}}(\gamma)$ $\in_{\gamma}$-below $\mathsf{K}$, such that $\delta_\mathsf{P}<\bar{\varepsilon}^{\gamma}_i<\delta_\mathsf{K}$.
Suppose $\one\nolimits_{\gamma}\forces\bar{\varepsilon^{\gamma}_i}\notin\acc^+(\nacc(C^{\gamma}_{\delta_\mathsf{K}}))$.
We need to prove $\one\nolimits_{\gamma}\forces\delta_\mathsf{P}\notin\acc^+(\nacc(C^{\gamma}_{\delta_\mathsf{K}}))$.

By our assumption, as $\one\nolimits_{\gamma}$ forces $C^{\gamma}_{\delta_{\mathsf K}}\cap\varepsilon^{\gamma}_i$ is bounded, then it is forced that
$C^{\gamma}_{\delta_\mathsf{K}}\cap[\bar{\Lambda}^{\gamma}_i,\bar{\varepsilon}^{\gamma})$ is empty.
As $q$ is appropriate for $p$, $\one\nolimits_{\gamma}$ forces that $\delta_{\mathsf P}>\Lambda^{\gamma}_i$.
But then $\one\nolimits_{\gamma}\forces\delta_P\notin\acc^+(\nacc(C^{\gamma}_{\delta_{\mathsf K}}))$. As desired.
\end{proof}

Our first aim is to show that such condition $q$ is attainable from $p$.
In order to establish this we will need the following definition:
\begin{defn}
We say that a set $\Delta$ is \emph{unbounded in $\prod_{i<n, j<m}\delta^j_{i}$} iff for all matrices $\langle\varsigma^j_i\mid i<n,\, j<m\rangle\in\prod_{i<n, j<m}\delta^j_i$ there exists a matrix $\langle\eta^{j}_i\mid i<n,j<m\rangle\in\Delta$ such that $\varsigma^j_i<\eta^j_i<\delta^j_i$ for every $i<n$ and $j<m$.
\end{defn}

Now we arrive at the main lemma of this subsection.

\begin{lemma}\label{finding appropriate}
Given $\alpha\in E$, for $p\in\mathbb O_{\alpha}$, suppose $\mathsf{M}\in\pi_0(\mathcal M_p)\cap\mathcal B_{st}$. Let $D\in \mathsf{M}$ some dense set.
Then, there exists a $q\in\mathbb O_{\alpha}\cap \mathsf{M}$ which is appropriate for $p$.
\end{lemma}
\begin{proof}
Without loss of generality, $p\in D$.
Recall that $\xi_0\in\dom(F_p)\cap \mathsf{M}$.
For $\gamma\in\dom(F_p)\cap \mathsf{M}$, let $\vec{\varepsilon}^{\gamma}:=\langle\varepsilon^{\gamma}_i\mid i\leq m^{\xi_0}_{p}+1\rangle$ (with repetitions) the set of the critical cofinality $\varepsilon$-heights and $\epsilon^{\gamma}_{m_p^{\xi_0}}:=\kappa$.
Set $F:=\dom(F_{p\restriction \mathsf{M}})$ and $n:=m_{p}^{\xi_0}$.

For each $i\leq n$ and $\gamma\in F$, set $\mathsf{N}_i\in\mathcal M_{p\res \mathsf{M}}(\gamma)$ such that $\delta_{\mathsf{N}^{\gamma}_i}=\varepsilon^{\gamma}_i$.
Let $\varsigma^{\gamma}_i$ be the first element in $\mathsf{M}\cap \mathsf{N}^{\gamma}_i\rest\gamma\cap V_{\kappa}$ above $\max(\Delta(\mathcal M_{p\res \mathsf{M}}(\gamma))\cap\bar{\varepsilon}^{\gamma}_i)$.

Now, define $\Delta$ to be the set of all matrices $\langle\eta^{\gamma}_i\mid i\leq n,\gamma\in F\rangle$ such that, there exists a $q\in D$ extending $p\res \mathsf{M}$ with the following properties:
\begin{itemize}
\item For all $\gamma\in F$ and $i<n$, $\Delta(\mathcal M_q(\gamma))\setminus\Delta^{\gamma}_{p\restriction \mathsf{M}}\cap[\varepsilon_{i+1}^\gamma,\varepsilon_i^\gamma]\s(\eta^{\gamma}_i,\varepsilon^{\gamma}_i]$.
\item For $i=n$, $\Delta(\mathcal M_q(\gamma))\setminus\Delta^{\gamma}_{p\res \mathsf{M}}\cap(\varepsilon^{\gamma}_{i-1},\kappa)\s(\eta^{\gamma}_i,\kappa)$.
\end{itemize}

Note that $\Delta$ is definable in $\mathsf{M}$.

\begin{claim}
$\Delta$ is unbounded in $\prod_{i\leq n+1,\gamma\in F}\varepsilon^{\gamma}_i$.
\end{claim}
\begin{proof}
Suppose not. Then, as $\Delta$ is definable in $\mathsf{M}$ there exists a sequence $\langle\eta^{\gamma}_i\mid i\leq n,\gamma\in F\rangle$ in $\mathsf{M}$, such that for all $q\in D$ there is $\gamma_q\in F$ and $i_q\leq n$ with $\sup(\Delta(\mathcal M_q(\gamma_q))\setminus\varsigma^{\gamma}_{i_q})<\eta^{\gamma_q}_{i_q}$.

In particular, the conclusion holds for $p$.
But then, we get that $\delta_{\mathsf{M}\wedge \mathsf{N}\rests\gamma_p}\leq\eta^{\gamma_p}_i\in \mathsf{N}\wedge \mathsf{M}\rest\gamma_p$ in case $i_p<n$.
And $\delta_\mathsf{M}<\eta^{\gamma_p}_{i_p}\leq\eta^{\gamma_p}\in \mathsf{N}\wedge \mathsf{M}\rest\gamma_p$ in case $i_p=n$.
This is a contradiction.\qedhere
\end{proof}

Since $\Delta$ is an unbounded set in $\mathsf{M}$, and as $\langle\Lambda^{\gamma}_i\mid\gamma\in F,i\leq n\rangle$ are defined in $\mathsf{M}$, by elementarity of $\mathsf{M}$ there exists $\langle\eta^{\gamma}_i\mid \gamma\in F,i\leq n\rangle$ in $\Delta\cap \mathsf{M}$ such that for every $i\leq n$ and $\gamma\in F$, $\Lambda^{\gamma}_i<\eta^{\gamma}_i$.
Let $q\in D\cap \mathsf{M}$ a condition witnessing this.

Now, let $\gamma\in F$. Given $\mathsf{K}\in\mathcal M_p(\gamma)$ and $\mathsf{P}\in\mathcal M_q(\gamma)$ with $\delta_\mathsf{P}<\bar{\varepsilon}^{\gamma}_i<\delta_{\mathsf{K}}$.
Then, by the choice of $\eta^{\gamma}_i$ we get because $\delta_\mathsf{P}\in\Delta(\mathcal M_q)\setminus\varsigma^{\gamma}_i$, $\delta_p>\eta^{\gamma}_i>\Lambda^{\gamma}_i$.
As desired.
\end{proof}

\subsection{Countable properness}

First, as usual we must prove that for each $p\in\mathbb O_{\alpha}$ for some $\alpha\in E$ and $\mathsf{M}\in\mathcal B_{st}$ with $p\in \mathsf{M}$ we may extend to a condition we will later show is a master condition.

\begin{lemma}
Let $p\in\mathbb O_{\alpha}$ and $\mathsf{M}\in\mathcal B_{st}$ such that $p\in \mathsf{M}$. Then there is a condition $p_\mathsf{M}\leq p$ such that $\mathsf{M}\in p_\mathsf{M}$ and $\delta_\mathsf{M}\in D^{\gamma}_{p_\mathsf{M}}$ for all $\gamma\in\dom(F_{p_\mathsf{M}})$.
\end{lemma}
\begin{proof}
Set $p_\mathsf{M}:=\langle\mathcal M_p\wedge\{\mathsf{M}\}, d_p, F_q\rangle$ where:
\begin{enumerate}
\item $\dom(F_{p_\mathsf{M}})=\dom(F_p)$;
\item For $\gamma\in\dom(F_{p_\mathsf{M}})$ if $\gamma$ is good, then set $F_{p_\mathsf{M}}(\gamma):=(B^{\gamma}, \langle \tau^{\gamma}_{\delta}\mid \delta\in\Delta_1(D^{\gamma})\rangle)$ where:
\begin{itemize}
\item $D^{\gamma}:=D^{\gamma}_p\cup\{\delta_{\mathsf{N}\wedge \mathsf{M}},\delta_\mathsf{M}\mid \mathsf{N}\in\pi_1(\mathcal M_p(\gamma))\}$;
\item $B^{\gamma}:=B^{\gamma}_p\cup\{\delta_{\mathsf M}\}$;
\item $\tau^{\gamma}_{\delta}:=\tau^{\gamma}_{\delta,p}$ for $\delta\in\Delta_1(D^{\gamma}_p)$.
\end{itemize}
Otherwise, $\gamma$ is not good, and then we set $F_q:=F_p$.
\end{enumerate}

Note that verifying that indeed $p_\mathsf{M}\leq p$ and Clauses (a),(b1)--(b3) is clear.
Thus it is only left to show that $p_\mathsf{M}$ satisfies Clause (b4).
Let $\gamma\in\dom(F_p)$, $\mathsf{K}\in\pi_1(\mathcal M_{p_\mathsf{M}}(\gamma))$
and $\mathsf{P}\in\pi_0(\mathcal M_{p_\mathsf{M}}(\xi_0))$.
Since $\pi_1(\mathcal M_{p_\mathsf M})=\pi_1(\mathcal M_p)$, $\mathsf K\in\pi_1(\mathcal M_p)$.
Suppose $\mathsf P$ is such that $\delta_{\mathsf P}\in B^{\gamma}_p$
we finish by the fact $p$ is condition.
Otherwise $\mathsf P$ is of the right form for $\mathsf K$ in $p_{\mathsf M}$
If we are able to show it is of the right form for $\mathsf K$ in $p$ we are done.
We divide into two cases:
\begin{itemize}
\item[$\br$] If $\mathsf P\neq(\mathsf K\rest\xi_0)\wedge\mathsf S$ for any $\mathsf S\in\mathcal M_{p_{\mathsf M}}(\xi_0)$, then it is of the right form for $\mathsf K$ in $p$.
\item[$\br$] Else, as $\mathsf P$ is of the right form for $\mathsf K$ in $p_{\mathsf M}$,
There exist $\mathsf N\in\pi_1(\mathcal M_{p}(\xi_0))$ $\in_{\xi_0}$
above $\mathsf K$ and $\mathsf S\in\pi_0(\mathcal M_{p_{\mathsf M}}(\xi_0))$
such that $\mathsf P=\mathsf N\wedge\mathsf S$.
If $\mathsf S\in\pi_0(\mathcal M_p(\xi_0))$ then $\mathsf P$ is of the right form for $\mathsf K$ in $p$.
Otherwise, $\mathsf S=\mathsf N'\wedge(\mathsf M\rest\xi_0)$ for some $\mathsf N'\in\pi_1(\mathcal M_p)$.
As $\mathsf S$ is $\in_{\xi_0}$-above $\mathsf K\rest\xi_0$ and $\mathsf K\rest\xi_0\in\mathsf M$,
we have that $\mathsf K\rest\xi_0\in_{\xi_0}\mathsf S$.
So, $\mathsf K\rest\xi_0\in_{\xi_0} \mathsf P$ this is acontradiction to the assumption $\mathsf P$ is $\in_{\xi_0}$-below $\mathsf K$.
\end{itemize}

\end{proof}

Given $\alpha\in E$, for $r\in\mathbb O_{\alpha}\cap \mathsf{M}$ where $\mathsf{M}\in\mathcal B_{st}$, in the previous Lemma we have showed that one may extend $r$ to a condition $r_\mathsf{M}$ such that, $\mathsf{M}\in\mathcal M_{r_\mathsf{M}}$.
Through the rest of this subsection we shall show that $r_\mathsf{M}$ is indeed a master condition.
We again remind the reader that the proof will go by induction.

In order to achieve our goal, take any $p\leq r_\mathsf{M}$.
Denote $\beta:=\max(\dom(F_p)\cap \mathsf{M})$.
By the induction hypothesis, as $q\res\beta$ is appropriate for $p\res\beta$, there exists $t'\in\mathbb O_{\beta}$ such that $\bar{t}\leq q\res\beta,p\res\beta$.

Applying Lemma~\ref{appropriate lemma} there exists $q\in\mathbb O_{\alpha}\cap \mathsf{M}$ extending $p\res \mathsf{M}$ appropriate for $p$.

We define $t:=\langle\mathcal M_t,d_t,F_t\rangle$ as follows:
\begin{enumerate}
\item $\langle\mathcal M_t,d_t\rangle:=\langle\mathcal M_p,d_p\rangle\wedge\langle\mathcal M_q,d_q\rangle$.
\item $\dom(F_t)=\dom(F_{\bar{t}})\cup(\dom(F_p)\setminus\beta)\cup(\dom(F_q)\setminus\beta)$.
\item For each $\gamma\in\dom(F_{t})\cap\beta$, $F_t(\gamma)=F_{\bar{t}}(\gamma)$.
\item For each $\gamma\in\dom(F_t)\setminus\beta$ and $\mathsf{K}\in\mathcal M_t(\gamma)$ we set
\begin{itemize}
\item $D^{\gamma}_t:=D^{\gamma}_p\cup D^{\gamma}_q$;
\item $B^{\gamma}_t:=\begin{cases} B^{\gamma}_{\bar{t}}& \gamma<\beta;\\
B^{\gamma}_q& \gamma\in\mathsf M\setminus\beta;\\
B^{\gamma}_p& \text{otherwise.}

\end{cases}$
\item If $\gamma\in\dom(F_q)$ and $\mathsf{K}\in\pi_1(\mathcal M_p(\gamma))$ or $\gamma\in\dom(F_p)$ and $\mathsf{K}\in\mathcal M_q(\gamma)$, given $\mathsf{P}\in\pi_0(\mathcal M_t(\xi_0))$ $\in_{\xi_0}$-below $\mathsf{K}$
set
$$\tau^{\gamma}_{\mathsf{K},\mathsf{P}}:=\sup\{\tau<\omega_1\mid\exists t'\in\mathbb O_{\gamma}(t'\leq t\res\gamma)\,\&\,(t'\forces\dot{h}^{\gamma}_{\delta}[\delta_\mathsf{P}]\s\tau)\} $$

and for $\mathsf{N}\in\pi_1(\mathcal M_t(\xi_0))$ $\in_{\xi_0}$-below $\mathsf{K}$,
$$\tau^{\gamma}_{\mathsf{K},\mathsf{N}}:=\sup\{\tau<\omega_1\mid\exists t'\in\mathbb O_{\gamma}(t'\leq t\res\gamma)\,\&\,(t'\forces\dot{h}^{\gamma}_{\delta}(\delta_\mathsf{N})\in\tau)\} $$

As we already know that $\mathbb O_{\gamma}$ is proper, by Remark~\ref{proper remark} both $\tau^{\gamma}_{\mathsf{K},\mathsf{P}}$ and $\tau^{\gamma}_{\mathsf{K},\mathsf{N}}$ are strictly less than $\omega_1$.
Thus, the following is a well defined countable ordinal for $\gamma\in \dom(F_t)$ and $\mathsf{K}\in\pi_1(\mathcal M_t(\gamma)$:

$$\tau^{\gamma}_{\delta_\mathsf{K}}:=\begin{cases}\min(\omega_1\setminus\bigcup\{\tau^{\gamma}_{\mathsf{K},\mathsf{P}}\mid \mathsf{P}\in\mathcal M_t(\xi_0)\}) & \gamma\in\dom(F_q), \mathsf{K}\in\mathcal M_p;\\
\text{same as above}& \gamma\in\dom(F_p), \mathsf{K}\in\mathcal M_q;\\
\tau^{\gamma}_{\delta_\mathsf{K}} &\text{otherwise}

\end{cases}$$

\item For $\gamma\in\dom(F_p)$, $\mathsf{K}\in\pi_1(\mathcal M_q(\gamma))$ we define $\tau^{\gamma}_\mathsf{K}$ similarly.

\end{itemize}
\end{enumerate}

We next present an argument which will appear several times throughout this paper:

\begin{lemma}\label{crucial claim}
Suppose $\mathsf{K}\in\pi_1(\mathcal M_p(\gamma))\setminus\mathcal M_q(\gamma)$.
Let $i<m_p$ such that $\bar{\varepsilon}^{\xi_0}_i<\delta_{\mathsf K}$.
If $\one\nolimits_{\gamma}\nforces\bar{\varepsilon}^{\xi_0}_i\notin\acc^+(\nacc(C^{\gamma}_{\delta_\mathsf{K}}))$, then
$$t\res\gamma\forces\tau^{\gamma}_{\delta_\mathsf{K}}\notin h^{\gamma}_{\delta_\mathsf{K}}[\bar{\varepsilon}^{\xi_0}_{i}].$$
\end{lemma}
\begin{proof}

There are two possible cases for $\bar{\varepsilon}^{\xi_0}_i$ which may occur:
\begin{itemize}
\item[(i)] Suppose $\bar{\varepsilon}^{\xi_0}_i=\delta_\mathsf{M}$.
In this case as $\delta_\mathsf{M}\in B^{\gamma}_t$ and we have that $\delta_\mathsf{K}>\bar{\varepsilon}^{\xi_0}_i$,
by Clause (b4) applied to $p$:
$$p\res\gamma\forces\tau^{\gamma}_{\delta_\mathsf{K}}\notin h^{\gamma}_{\delta_\mathsf{K}}[\delta_\mathsf{M}].$$
By our induction hypothesis $t\res\gamma\leq p\res\gamma$,
$$t\res\gamma\forces\tau^{\gamma}_{\delta_\mathsf{K}}\notin h^{\gamma}_{\delta_\mathsf{K}}[\delta_\mathsf{M}].$$

\item[(ii)] Else $\bar{\varepsilon}^{\xi_0}_i<\delta_\mathsf{M}$.
As $\mathsf{K}\in\pi_1(\mathcal M_p(\gamma))\setminus\mathcal M_{p\res\mathsf M}(\gamma)$, again by Fact~\ref{fact 2.27} we may find $\mathsf{N}\in\pi_1(\mathcal M_{p\res \mathsf{M}}(\xi_0))$ such that $\bar{\varepsilon}^{\xi_0}_i=\delta_{\mathsf{N}\wedge (\mathsf{M}\rests\xi_0)}$.
In particular, as $\delta_{\mathsf M}\in B^{\gamma}_p$, $N\wedge(\mathsf M\rest\xi_0)$ is of the right form for $\mathsf K$.

So in this case we may apply Clause (b4) to $t\res\gamma$ and deduce:
$$t\res\gamma\forces\tau^{\gamma}_{\delta_\mathsf{K}}\notin h^{\gamma}_{\delta_\mathsf{K}}[\bar{\varepsilon}^{\xi_0}_i].$$

\end{itemize}
\end{proof}

We aim to show that $t$ is a condition extending both $p$ and $q$.
First note that the fact $t\leq p,q $ and $t$ satisfies Claues (a),(b1)--(b3) follows immediately
from the definition. Therefore
in order to accomplish this task, we will need the following lemmata:

\begin{lemma}
Let $\alpha\in E$. For $p\in\mathbb O_{\alpha}$ extending $r_\mathsf{M}$, let $q$ be appropriate for $p$.
Moreover, suppose $\gamma\in\dom(F_t)$ is such that $t\res\gamma$ is a legitimate condition extending $p\res\gamma$, $q\res\gamma$. Then $t\res(\gamma+1)$ satisfies Clause~(b4).
\end{lemma}
\begin{proof} There are two cases to consider:

\underline{Case 1:} Suppose $\gamma\in\dom(F_p)\setminus\dom(F_q)$. Given $\mathsf{K}\in\pi_1(\mathcal M_t(\gamma))$ let $\mathsf{P}\in\mathcal M_t(\xi_0)$ countable with $\delta_\mathsf{P}\in D^{\gamma}_t$ which is of the right form for $\mathsf K$ or $\delta_{\mathsf P}\in B^{\gamma}_t$ First note that all the countable models in $\mathcal M_t(\xi_0)$ have the following form:
\begin{enumerate}
\item[($\alpha$)] $\mathsf{P}\in\pi_0(\mathcal M_p(\xi_0))\setminus\mathcal M_q(\xi_0)$.
\item[($\beta$)] $\mathsf{P}\in\pi_0(\mathcal M_q(\xi_0))$.
\item[($\gamma$)] $\mathsf{P}=\mathsf W\wedge\mathsf S$ where $\mathsf W\in\pi_1(\mathcal M_q(\xi_0))$ and $\mathsf S\in\pi_0(\mathcal M_p(\xi_0))$.

\end{enumerate}
Therefore there are now several cases to consider:
\begin{itemize}
\item[$\br$] Suppose $\mathsf{K}\in\pi_1(\mathcal M_q(\gamma))$. The conclusion in this case follows by our definition of $\tau^{\gamma}_{\delta_\mathsf{K}}$.

\item[$\br$] Else, $\mathsf{K}\in\pi_1(\mathcal M_p(\gamma))\setminus\mathcal M_q(\gamma)$.

Note that as $\delta_\mathsf{M}\in D^{\gamma}$, $\{\bar{\varepsilon}^{\xi_0}_i\mid i<m_p\}\s D^{\gamma}$ by Clause (b3) applied to $p$.

Now as $\mathsf{P}$ can be of three possible forms, there are three cases to consider:

$\br\br$ If $\mathsf{P}$ is of form $(\alpha)$,
then by Fact~\ref{Timer fact} $\mathsf P$ is of the right form for $\mathsf K$ in $p$.
So the conclusion holds by the fact $p$ is a condition satisfying Clause (b4).

$\br\br$ If $\mathsf{P}$ is of the form $(\beta)$, let $i<m_p$ be the last such that $\delta_p<\bar{\varepsilon}^{\xi_0}_i$. then we still have two more subcases.

$\br\br\br$ Suppose $\one\nolimits_{\gamma}\nforces\bar{\varepsilon}^{\xi_0}_i\notin\acc^+(\nacc(\dot{C}^{\gamma}_{\delta_\mathsf{K}}))$,
then as $\delta_\mathsf{P}<\bar{\varepsilon}^{\xi_0}_i$ it follows that,
$$t\res\gamma\forces \dot{h}^{\gamma}_{\delta}[\delta_\mathsf{P}]\s h^{\gamma}_{\delta}[\bar{\varepsilon}^{\xi_0}_{i}]$$
By Lemma~\ref{crucial claim}, the conclusion holds.

$\br\br\br$ Else, $\one\nolimits_{\gamma}$ forces $\bar{\varepsilon}^{\xi_0}_i\notin\acc^+(\nacc(C^{\gamma}_{\delta_\mathsf{K}})))$.
Then by Lemma~\ref{appropriate lemma} Clause (1), $\delta_\mathsf{P}$ is forced to between two consecutive elements of $\nacc(C^{\gamma}_{\delta_\mathsf{K}}))$. So, the conclusion holds trivially.

$\br\br$ Otherwise, $\mathsf{P}$ is of form $(\gamma)$.
Let $\mathsf W\in\mathcal M_q(\xi_0)$ and $\mathsf S\in\mathcal M_p(\xi_0)$ such that $\mathsf{P}=W\wedge S$.
By Fact~\ref{Timer fact}, there exists $i<m_p$ such that,
$$\varepsilon^{\xi_0}_{i+1}<\delta_\mathsf{P}<\delta_{\mathsf W}<\bar{\varepsilon}^{\xi_0}_i<\delta_{\mathsf S}<\varepsilon^{\xi_0}_i.$$

There are two cases to consider:

$\br\br\br$ If $\one\nolimits_{\gamma}\nforces\bar{\varepsilon}^{\xi_0}_i\in\acc^+(\nacc(C^{\gamma}_{\delta_\mathsf{K}}))$,
Then as $\delta_\mathsf{P}<\bar{\varepsilon}^{\xi_0}_i$, $p$ is a condition and $t\res\gamma\leq p\res\gamma$:
$$t\res\gamma\forces \dot{h}^{\gamma}_{\delta_\mathsf{K}}[\delta_\mathsf{P}]\cap\{\tau^{\gamma}_{\delta_\mathsf{K}}\}=\emptyset.$$

$\br\br\br$ Else, $\one\nolimits_{\gamma}\forces\bar{\varepsilon}^{\xi_0}_i\notin\acc^+(\nacc(C^{\gamma}_{\delta_\mathsf{K}}))$.
Since $q$ is appropriate for $p$ we have $\delta_{\mathsf W}>\Lambda^{\xi_0}_i>\bar{\Lambda}^{\xi_0}_i$.
In addition, $\delta_W\in\mathsf S$. So, by the choice of $\Lambda^{\xi_0}_i$, $\sup(\bar{\varepsilon}^{\xi_0}\cap\mathsf S)=\bar{\varepsilon}^{\xi_0}_i$.

Since we chose $\Lambda^{\gamma}_i$ to be above $\min(\mathsf S\setminus\bar{\Lambda}^{\gamma}_i)$, in particular
$\bar{\Lambda}^{\gamma}_i\leq\delta_{\mathsf{P}}$.
But, $p\res\gamma\forces \nacc(\dot{C}^{\gamma}_{\delta_\mathsf{K}})\cap[\bar{\Lambda}^{\gamma}_i,\varepsilon^{\gamma}_{i}]=\emptyset$.

Thus as $t\res\gamma\leq p\res\gamma$, the condition holds trivially.

\end{itemize}

\underline{Case 2:} Suppose $\gamma\in\dom(F_q)$.
Given $\mathsf{K}\in\pi_1(\mathcal M_t(\gamma))$, let $\mathsf{P}\in\mathcal M_t(\xi_0)$ countable with $\delta_\mathsf{P}\in D^{\gamma}$ which is the right form for $\mathsf{K}$ or $\delta_{\mathsf P}\in B^{\gamma}_t$.
We again conclude that $\mathsf{P}$ is of one of the forms ($\alpha$)--($\gamma$) above.
So again we separate into cases:
\begin{itemize}
\item[$\br$] Suppose $\mathsf{K}\in\pi_1(\mathcal M_p(\gamma))\setminus\mathcal M_q(\gamma)$.
Then the conclusion follows from the definition of $\tau^{\gamma}_\mathsf{K}$.

\item[$\br$] Else, $\mathsf{K}\in\pi_1(\mathcal M_q(\gamma))$.

$\br\br$ If $\mathsf{P}$ is of form $(\beta)$. In case $\mathsf P\in B^{\gamma}_t$ we are finished.
Thus, assume that $\mathsf P$ is of the right form for $\mathsf K$ in $t$.
We need to show that $\mathsf P$ is of the right form for $\mathsf K$ in $q$.
Suppose not. Then there exist $\mathsf S\in\pi_0(\mathcal M_q(\xi_0))$ such that $\mathsf P=\mathsf K\rest\xi_0\wedge\mathsf S$ and there are $\mathsf N\in\pi_1(\mathcal M_t(\xi_0)$, $\mathsf R\in\pi_0(\mathcal M_t(\xi)$ such that,
$$P=N\wedge R\in_{\xi_0}\mathsf K\rest\xi_0\in_{\xi_0}N\in_{\xi_0} R.$$

We may assume that $\mathsf R\in\mathcal M_p(\xi_0)\setminus\mathcal M_q(\xi_0)$ as otherwise both $\mathsf N,\mathsf R\in\mathsf M$ and thus $\mathsf P$ is of the right form for $\mathsf K$ in $q$, contradicting our assumption.
but, recall that in addition $\delta_{\mathsf R}\in B^{\gamma}_q$ so there is $\mathsf Q\in\pi_0(\mathcal M_q(\xi_0))$ such that $\delta_{\mathsf Q}=\delta_{\mathsf R}$.
This is a contradiction to the fact $\mathcal M_p(\xi_0)\cup\mathcal M_q(\xi_0)$ form an $\in_{\xi_0}$-chain.

$\br\br$ Suppose $\mathsf{P}$ is of form $(\alpha)$.
Since $\mathsf{K}$ is in $\mathsf{M}$ in particular, $\mathsf{K}$ is $\in_{\xi_0}$ below $\mathsf{M}$.
Therefore there exists $\mathsf{N}\in\pi_1(\mathcal M_{p\rests \mathsf{M}}(\xi_0))$ such that $\mathsf{P}\in_{\xi_0} \mathsf{N}\in_{\xi_0} \mathsf{K}$.
Denote $\dot{\delta}^+:=\min(\nacc(\dot{C}^{\gamma}_{\delta_\mathsf{K}})\setminus\delta_{\mathsf{N}})$, $\dot{\delta}^-:=\max(\nacc(\dot{C}^{\gamma}_{\delta_\mathsf{K}})\cap\dot{\delta}^+)$.
Since $\gamma\in \mathsf{M}$, all of the relevant objects are defined in $\mathsf{M}$ as names.
Therefore,
$$\one\nolimits_{\gamma}\forces \delta^-<\delta_{\mathsf{N}\wedge \mathsf{M}\rests\xi_0}\leq\delta_{\mathsf{P}}<\delta^+,$$
so,
$$\one\nolimits_{\gamma}\forces \delta_\mathsf{P}\notin\acc^+(\nacc(C^{\gamma}_{\delta_\mathsf{K}})).$$

$\br\br$ Else, $\mathsf{P}$ is of form $(\gamma)$.
Let $\mathsf S\in\pi_0(\mathcal M_p(\xi_0))$ and $\mathsf W\in\pi_1(\mathcal M_q(\xi_0))$.

The following is well-known.
\begin{claim}
$\delta_{\mathsf W\wedge\mathsf S}\geq\delta_{\mathsf W\wedge\mathsf M}$.
\end{claim}
\begin{proof}
We prove the claim by induction on $\mathcal M_t(\xi_0)$ with respect to $\in^{*}_{\xi_0}$.
Since $\mathsf S\notin\mathsf M$ there are two possibilities:
\begin{itemize}
\item[(i)] $\mathsf S$ in $\in_{\xi_0}$-above $\mathsf M$;
\item[(ii)] There exist $\mathsf N\in\mathcal M_{p\res\mathsf M}(\xi_0)$ such that $\mathsf S\in[\mathsf N\wedge\mathsf M,\mathsf N)^{\xi_0}_{\mathcal M_p}$.
\end{itemize}

Suppose that $\mathsf S$ satisfies (i).
In case there is no Magidor model in the interval $(\mathsf M,\mathsf S)^{\xi_0}_{\mathcal M_p}$ we have that $\mathsf M\in_{\xi_0}\mathsf S$.
Therefore as $\mathsf W\wedge\mathsf M\in_{\xi_0}\mathsf M$, $\mathsf W\wedge\mathsf M\in_{\xi_0}\mathsf S$.
Recalling that $\mathsf W\wedge\mathsf M\in_{\xi_0}\mathsf W$ the conclusion follows.

Now suppose there is a Magidor model in $(\mathsf M,\mathsf S)^{\xi_0}_{\mathcal M_p}$.
Let $\mathsf R$ be the $\in^*_{\xi_0}$ largest such model.
As $\mathsf R\in_{\xi_0}\mathsf S$, we have that $\mathsf R\wedge\mathsf S$ is defined
and $\mathsf W\in_{\xi_0}\mathsf R\wedge\mathsf S$.
Therefore by our induction hypothesis, $\delta_{\mathsf W\wedge(\mathsf R\wedge\mathsf S)}\geq\delta_{\mathsf W\wedge\mathsf M}$.
But, $\mathsf W\wedge(\mathsf R\wedge\mathsf S)=\mathsf W\wedge\mathsf S$.

In case $\mathsf S$ satisfies (ii) we have a similar proof.
\end{proof}

Denote $\dot{\delta}^+:=\min(\nacc(\dot{C}^{\gamma}_{\delta_\mathsf{K}})\setminus\delta_{W})$, $\dot{\delta}^-:=\max(\nacc(\dot{C}^{\gamma}_{\delta_\mathsf{K}})\cap\dot{\delta}^+)$.
Since $\gamma\in \mathsf{M}$, all of the relevant objects are defined in $\mathsf{M}$ as names.
Therefore,
$$\one\nolimits_{\gamma}\forces \delta^-<\delta_{W\wedge(\mathsf{M}\rests\xi_0)}\leq\delta_{\mathsf{P}}<\delta^+$$
so,
$$\one\nolimits_{\gamma}\forces \delta_\mathsf{P}\notin\acc^+(\nacc(C^{\gamma}_{\delta_\mathsf{K}})).$$

\end{itemize}
The proof is now complete.
\end{proof}
\begin{lemma}
Let $\alpha\in E$. For $p\in\mathbb O_{\alpha}$, suppose $\mathsf{M}\in\pi_0(\mathcal M_p)\cap\mathcal B_{st}$ and let $q$ be appropriate for $p$.
Suppose $\gamma\in\dom(F_t)$ is such that $t\res\gamma$ is a legitimate condition. Then $t\res(\gamma+1)$ satisfies Clause (b1) holds for $t$.
\end{lemma}
\begin{proof}
There are two main cases to consider:

\underline{Case 1:} Suppose $\gamma\in\dom(F_p)\setminus\dom(F_q)$.

\begin{itemize}
\item[$\br$] Suppose $\mathsf{K}\in\pi_1(\mathcal M_q(\gamma))$. In this case the conclusion in this case follows by the definition of $\tau^{\gamma}_{\delta_\mathsf{K}}$.

\item[$\br$] Else, $\mathsf{K}\in\pi_1(\mathcal M_p(\gamma))\setminus\mathcal M_q(\gamma)$. Let $\mathsf{P}\in\pi_1(\mathcal M_t(\gamma))$ which is $\in_{\gamma}$-below $\mathsf{K}$.

$\br\br$ If $\mathsf{P}\in\mathcal M_p(\gamma)$, then $p\res\gamma\forces h^{\gamma}_{\delta_\mathsf{K}}(\delta_\mathsf{P})\neq\tau^{\gamma}_\mathsf{K}$ by the fact Clause (b1) holds for $p$.

$\br\br$ If $\mathsf{P}\in\mathcal M_q(\gamma)$.
Going back to $\xi_0$ as $\delta_{\mathsf{P}\res\xi_0}=\delta_\mathsf{P}$, we may take $i<m_p$ to be the last such that $\delta_\mathsf{P}<\bar{\varepsilon}^{\xi_0}_i$.

Note that, as $\delta_\mathsf{M}\in D^{\gamma}$, $\bar{\varepsilon}^{\xi_0}_i\in D^{\gamma}$ by Clause (b3) applied to $t\res\gamma$.
Now there are two cases to consider:

$\br\br\br$ If $\one\nolimits_{\gamma}\nforces\bar{\varepsilon}^{\xi_0}_i\notin\acc^+(\nacc(\dot{C}^{\gamma}_{\delta_\mathsf{K}}))$.
Then by Lemma~\ref{crucial claim} as $\delta_\mathsf{P}<\bar{\varepsilon}^{\xi_0}_i$ we get,
$$t\res\gamma\forces h^{\gamma}_{\delta_\mathsf{K}}[\delta_{\mathsf{P}}]\cap\{\tau^{\gamma}_{\delta_\mathsf{K}}\}=\emptyset.$$

$\br\br\br$ Else, $\one\nolimits_{\gamma}\forces\bar{\varepsilon}^{\xi_0}_i\notin\acc^+(\nacc(\dot{C}^{\gamma}_{\delta_\mathsf{K}}))$.

Then as $q$ is appropriate for $p$, by Lemma~\ref{appropriate lemma} Clause (1),
$t\res\gamma\forces\delta_\mathsf{P}\notin\nacc(\dot{C}^{\gamma}_{\delta_{\mathsf K}})$.
\end{itemize}

\underline{Case 2:} Suppose $\gamma\in\dom(F_q)\setminus\beta$.
\begin{itemize}
\item[$\br$] Suppose $\mathsf{K}\in\pi_1(\mathcal M_p(\gamma))$. The conclusion in this case follows by the definition of $\tau^{\gamma}_{\delta_\mathsf{K}}$.

\item[$\br$] Else, $\mathsf{K}\in\pi_1(\mathcal M_q(\gamma))\setminus\mathcal M_p(\gamma)$. Let $\mathsf{P}\in\pi_1(\mathcal M_t(\gamma))$ below $\mathsf{K}$.

$\br\br$ If $\mathsf{P}\in\mathcal M_q(\gamma)$, then the conclusion holds by the fact $q$ is a condition.

$\br\br$ Else, $\mathsf{P}\in\mathcal M_p(\gamma)\setminus\mathcal M_q(\gamma)$.
Then there exists $\mathsf{N}\in\pi_1(\mathcal M_{p\res \mathsf{M}}(\gamma))$ such that, $\mathsf{P}\in_{\gamma} \mathsf{N}\in_{\gamma} \mathsf{K}$.
Because $\gamma\in \mathsf{M}$, note that $(\mathsf{N}\wedge \mathsf{M})\res\gamma\in\mathcal M_p(\gamma)$.
So, $\mathsf{N}\wedge \mathsf{M}\in_{\gamma} \mathsf{P}$ since $\mathsf{P}\notin\mathcal M_q(\gamma)$.
Let $\dot{\delta}^+:=\min(C^{\gamma}_{\delta_\mathsf{K}}\setminus\delta_\mathsf{N})$ and $\dot{\delta}^-$
be the predecessor of $\dot{\delta}^+$ in $C^{\gamma}_{\delta_\mathsf{K}}$.
Note that, as $\mathsf{N}, \mathsf{K}\in \mathsf{M}$, $\dot{\delta}^+\in \mathsf{M}$ and $\dot{\delta}^-\in \mathsf{N}\cap \mathsf{M}\cap V_{\kappa}$.
So again we get $t\res\gamma\forces\delta_\mathsf{P}\notin\nacc(\dot{C}^{\gamma}_{\delta_\mathsf{K}})$. \qedhere
\end{itemize}
\end{proof}

The following corollary follows immediately by the lemmata above.
\begin{cor}
Let $\alpha\in E$. For $r\in\mathbb O_{\alpha}$, suppose $\mathsf{M}\in\pi_0(\mathcal M_p)\cap\mathcal B^{\alpha}_{st}$ such that $r\in \mathsf{M}$. Let $p\leq r_\mathsf{M}$ and $q$ be appropriate condition for $p$.
Then $p$ and $q$ are compatible. \qed
\end{cor}

Putting everything together, we have verified the countable properness of our forcing.
\begin{thm}
Let $\alpha\in E$. Suppose $\mathsf{M}\in\mathcal B_{st}$ with $\eta(\mathsf{M})>\alpha$ and $\alpha\in \mathsf{M}$. Let $r$ be a condition in $\mathsf{M}$. then, $r_\mathsf{M}$ is $(\mathsf{M},\mathbb O_{\alpha})$-generic.

\end{thm}
\begin{proof}
We proceed by induction on $E$.
Note that $\mathbb O_{\min(E)}$ is $\mathbb B$-strongly proper.
Suppose now the conclusion holds for any ordinal in $E\cap\alpha$.
Suppose that $D\s\mathbb O_{\alpha}$ dense belonging to $\mathsf{M}$.
Let $p\leq r_\mathsf{M}$ be a condition in $D$.
By Lemma~\ref{finding appropriate} we may find $q\in D\cap \mathsf{M}$ extending $p\res \mathsf{M}$ which is appropriate for $p$.
By the Corollary above we are finished.\qedhere
\end{proof}

\subsection{Uncountable properness}

In this subsection we prove that $\mathbb O_{\alpha}$ is $\kappa$-proper by induction on $\alpha\in E$. We first introduce our candidate for a master condition.

Given $\alpha\in E$, for $p\in\mathbb O_{\alpha}$, let $\mathsf{M}\in\mathcal U^{\kappa}_{st}$ such that $p\in \mathsf{M}$.

Define $p_\mathsf{M}:=\langle \mathcal M_{p_\mathsf{M}},d_{p_\mathsf{M}},F_{p_\mathsf{M}}\rangle$ where:
\begin{itemize}
\item[$\br$] $\dom(F_{p_\mathsf{M}})=\dom(F_p)$;
\item[$\br$] For $\gamma\in\dom(F_{p_\mathsf{M}})$ if $\gamma$ is good, then set:
$$F_{r_\mathsf{M}}(\gamma):=(B^\gamma,\langle\tau^{\gamma}_{\delta}\mid\delta\in\Delta_1(D^{\gamma})\rangle),$$
where:

$\br\br$ $B^{\gamma}:=B^{\gamma}_p$.

$\br\br$ $D^{\gamma}:=D^{\gamma}_p\cup\{\delta_\mathsf{M}\}$.

$\br\br$ For $\gamma\in\dom(F_p)$, given $\mathsf{P}\in\pi_0(\mathcal M_p(\xi_0))$,
define
$$\tau^{\gamma}_{\mathsf{P}}:=\sup\{\tau<\omega_1\mid\exists q\in\mathbb O_{\gamma}(q\leq p_\mathsf{M}\res\gamma)\,\&\,(q\forces\dot{h}^{\gamma}_{\delta}[\delta_\mathsf{P}]\s\tau)\} $$

and for $\mathsf{N}\in\pi_1(\mathcal M_p(\xi_0))$,
$$\tau^{\gamma}_{\mathsf{N}}:=\sup\{\tau<\omega_1\mid\exists q\in\mathbb O_{\gamma}(q\leq p_\mathsf{M}\res\gamma)\,\&\,(q\forces\dot{h}^{\gamma}_{\delta}(\delta_\mathsf{N})\in\tau)\}$$

We already showed that $\mathbb O_{\gamma}$ is proper for every $\gamma\in E$, by Remark~\ref{proper remark} both $\tau^{\gamma}_{\mathsf{K},\mathsf{P}}$ and $\tau^{\gamma}_{\mathsf{K},\mathsf{N}}$ are countable ordinals.

Then, define $\tau^{\gamma}_{\delta_\mathsf{M}}:=\min(\omega_1\setminus\bigcup\{\tau^{\gamma}_\mathsf{P},\tau^{\gamma}_\mathsf{N}\mid \mathsf{N},\mathsf{P}\in\mathcal M_p(\gamma)\})$.
\end{itemize}

\begin{lemma}
Let $\alpha\in E$ and $\mathsf{M}\in\mathcal U_{st}$ with $p\in \mathsf{M}\cap\mathbb O_{\alpha}$. Then $p_\mathsf{M}$ is a legitimate condition extending $p$.
\end{lemma}
\begin{proof}
First let us show that $p_\mathsf M$ is a condition.
Clauses~(a) and (b1)--(b3) are clear from the definition of $p_{\mathsf M}$.
Therefore it is only left to show that $p_{\mathsf M}$ satisfies Clause (b4).
But for $\gamma\in\dom(F_{p_\mathsf M})$, $\mathsf K\in\pi_1(\mathcal M_{p_\mathsf M}(\gamma))$ and $\mathsf P\in\pi_0(\mathcal M_{p_\mathsf M}(\xi_0))$ as in the condition of Clause (b4), we have that either $\mathsf K=\mathsf M$,
which in this case the conclusion follows from the definition of $\tau^{\gamma}_{\mathsf M}$,
or $\mathsf K\in\mathcal M_r(\gamma)$ and then the conclusion follows
from the fact $r$ is a condition.

Secondly, note that the fact $p_M\leq p$ is immediate from the definition.
\end{proof}

Now, to see that $p_\mathsf{M}$ is indeed a master condition, given a dense set $D\in \mathsf{M}$ suppose $p\in D$ is some condition extending $p_\mathsf{M}$.
We shall find $q\in \mathsf{M}\cap D$ which extends $p\res \mathsf{M}$ and is compatible with $p$.
In order to establish that, we use an adaptation to argument from \cite{MR3356931}:

Let $\beta:=\dom(F_p)\cap\mathsf M$.
Note that $\beta\in\mathsf M$ and $\beta<\alpha$.
For each $\gamma\in\dom(F_p)\cap\beta$, set $\epsilon_{\gamma}:=\delta_{\mathsf{M}^{\gamma}}$ where $\mathsf{M}^{\gamma}$ is the predecessor of $\mathsf{M}$ in $\mathcal M_p(\gamma)$.
In addition fix an enumeration $\langle\gamma_i\mid i<n\rangle$ of $\dom(F_p)\cap\beta$.
Define $g:[\omega_2]^n\rightarrow[\omega_2]^{n}$ as follows:
for each tuple $\vec\tau=\langle\tau_i\mid i<n\rangle\in[\omega_2]^n$ such that $\tau_i>\epsilon_{\gamma_i}$ for every $i<n$,
let $g(\vec{\tau})$ be the least tuple $\langle\xi_i\mid i<n\rangle $ such that there exist a condition $q\in D$ such that,
\begin{enumerate}
\item $q\leq p\res \mathsf{M}$;
\item $q\res\beta=p\res\beta$;
\item for each $i<n$, $\xi_i>\tau_i$;
\item $\xi_i$ is the least ordinal in $\Delta(\mathcal M_q(\gamma_i))$ strictly above $\epsilon_{\gamma_i}$
\end{enumerate}

Note that $g$ is definable in $\mathsf{M}$.
Consider $C:=\{\eta<\omega_2\mid \bigcup g``[\eta]^n\s\eta\}$
which is a club and is also definable in $\mathsf{M}$.
As $\one_{\gamma}\forces\otp(C\cap\delta_{\mathsf{M}})=\delta_\mathsf{M}$, we may find some $\eta\in C\cap S\cap \mathsf{M}$ and $\tau<\eta$ such that,
$$\one\nolimits_{\gamma}\forces [\tau,\eta]\cap\bigcup\{\dot{C}^{\gamma'}_{\delta_\mathsf{K}}\mid \mathsf{K}\in\pi_1(\mathcal M_p(\gamma)),\gamma'\in\dom(F_p)\cap\gamma\}=\emptyset.$$

So, we may find a finite sequence $\vec{\tau}$ whose minimal element above $\max\{\xi_i\mid i<n\}$ and all elements in the interval $[\tau,\eta]$, together with a condition $q\in D\cap \mathsf{M}$ satisfying Clauses (1)--(4) above.
Thus, we may conclude that $\mathcal M_q(\gamma)\cap[\delta_{\mathsf{N}^{\gamma}},\delta_\mathsf{N}]\cap\bigcup_{\mathsf{K}\in\mathcal M_p(\gamma)}\dot{C}^{\gamma}_{\delta_\mathsf{K}}$ is forced to be empty, where $\mathsf{N}^{\gamma}$ is the $\in_{\gamma}$ predecessor of $\mathsf{N}$ in $\mathcal M_p(\gamma)$.
Moreover, Note that $h^{\gamma}_{\delta_\mathsf{M}}[\Delta(\mathcal M_q(\gamma))\cap[\delta_{\bar{\mathsf{M}}},\delta_{\mathsf{M}}]]$ is a constant which is distinct from $\tau^{\gamma}_{i_{\gamma}}$.

We define $t:=\langle\mathcal M_t,d_t,F_t\rangle$ as follows:
\begin{enumerate}
\item $\langle\mathcal M_t,d_t\rangle:=\langle\mathcal M_p,d_p\rangle\wedge\langle\mathcal M_q,d_q\rangle$.
\item $\dom(F_t)=\dom(F_p)\cup\dom(F_q)$.
\item For each $\gamma\in\dom(F_t)$, $F_t(\gamma):=(B^\gamma_t,\langle\tau^{\gamma}_\delta\mid\delta\in\Delta_1(D^{\gamma}_t)\rangle)$ where:
\begin{itemize}
\item[$\br$] $B^\gamma_t:=\begin{cases}B^{\gamma}_q&\gamma\in\dom(F_q);\\
B^{\gamma}_p& \gamma\notin\dom(F_q).
\end{cases}$

\item[$\br$] $D^{\gamma}_t:=D^{\gamma}_p\cup D^{\gamma}_q\cup\{\delta_{\mathsf W\wedge\mathsf S}\mid W\in\pi_1(\mathcal M_t),\, \mathsf S\in\pi_0(\mathcal M_t)\}$
\item[$\br$] For $\gamma\in\dom(F_q)$ and $\mathsf{K}\in\pi_1(\mathcal M_p(\gamma))$ or $\gamma\in\dom(F_p)$ and $\mathsf{K}\in\mathcal M_q(\gamma)$, given $\mathsf{P}\in\pi_0(\mathcal M_t(\xi_0))$ $\in_{\xi_0}$-below $\mathsf{K}$
set
$$\tau^{\gamma}_{\mathsf{K},\mathsf{P}}:=\sup\{\tau<\omega_1\mid\exists t'\in\mathbb O_{\gamma}(t'\leq t\res\gamma)\,\&\,(t'\forces\dot{h}^{\gamma}_{\delta}[\delta_\mathsf{P}]\s\tau)\} $$

and for $\mathsf{N}\in\pi_1(\mathcal M_t(\xi_0))$ $\in_{\xi_0}$-below $\mathsf{K}$,
$$\tau^{\gamma}_{\mathsf{K},\mathsf{N}}:=\sup\{\tau<\omega_1\mid\exists t'\in\mathbb O_{\gamma}(t'\leq t\res\gamma)\,\&\,(t'\forces\dot{h}^{\gamma}_{\delta}(\delta_\mathsf{N})\in\tau)\} $$

Again by Remark~\ref{proper remark} both $\tau^{\gamma}_{\mathsf{K},\mathsf{P}}$ and $\tau^{\gamma}_{\mathsf{K},\mathsf{N}}$ are strictly less $\omega_1$.
Thus, the following ordinal is well defined for $\gamma\in \dom(F_t)$ and $\mathsf{K}\in\pi_1(\mathcal M_t(\gamma))$:

$$\tau^{\gamma}_{\delta_\mathsf{K}}:=\begin{cases}\min(\omega_1\setminus\bigcup\{\tau^{\gamma}_{\mathsf{K},\mathsf{P}}\mid \mathsf{P}\in\mathcal M_t(\xi_0)\}) & \gamma\in\dom(F_q), \mathsf{K}\in\mathcal M_p;\\
\text{same as above}& \gamma\in\dom(F_p), \mathsf{K}\in\mathcal M_q;\\
\tau^{\gamma}_{\delta_\mathsf{K}} &\text{otherwise}.
\end{cases}$$

\end{itemize}
\end{enumerate}

Note that $t\in\text{pre}\mathbb O_{\alpha}$ and $t\leq p, q$.
The lemmata below shall prove that $t$ is indeed a legitimate condition extending both $p$ and $q$ by recursion on $\gamma$.

\begin{lemma}
Let $\alpha\in E$. For $p\in\mathbb O_{\alpha}$, suppose $\mathsf{M}\in\pi_1(\mathcal M_p)\cap\mathcal U_{st}$ and let $q\in\mathbb O_{\alpha}\cap \mathsf{M}$ extending $p\rest \mathsf{M}$ satisfying Clauses (1)--(4) above.
Suppose $\gamma\in\dom(F_t)$ is such that $t\res\gamma$ is a legitimate condition. Then $t\res(\gamma+1)$ satisfies Clause (b4) hold for $t$.
\end{lemma}
\begin{proof}
There are two cases to consider:

\underline{Case 1:} Suppose $\gamma\in\dom(F_p)\setminus\dom(F_q)$.
Given $\mathsf{K}\in\pi_1(\mathcal M_t(\gamma))$ and $\mathsf{P}\in\pi_0(\mathcal M_t(\xi_0))$ that is $\in_{\xi_0}$ below $\mathsf{K}$ with $\delta_\mathsf{P}\in D^{\gamma}_t$:
\begin{itemize}
\item[$\br$] Suppose $\mathsf{K}\in\pi_1(\mathcal M_q(\gamma))$. The conclusion in this case follows by the definition of $\tau^{\gamma}_{\delta_\mathsf{K}}$.

\item[$\br$] Else, $\mathsf{K}\in\pi_1(\mathcal M_p(\gamma))$.

$\br\br$ If $\mathsf{P}\in\mathcal M_{p}(\xi_0)$,
then the conclusion holds by the fact $p$ is a condition.

$\br\br$ Else, $\mathsf{P}\in\mathcal M_q(\xi_0)\setminus\mathcal M_p(\xi_0)$. Then as $\delta_\mathsf{P}(\xi_0)>\epsilon_{\xi_0}$, by the choice of $q$, $\delta_\mathsf{P}>\max\{\xi_i\mid i<n\}$.
Then, the conclusion holds trivially as $\one\nolimits_{\gamma}\forces\delta_p\notin\acc^+(\nacc(\dot{C}^{\gamma}_{\delta}))$.

\end{itemize}

\underline{Case 2:} Otherwise, $\gamma\in\dom(F_q)$. Given $\mathsf{K}\in\pi_1(\mathcal M_t(\gamma))$ and $\mathsf{P}\pi_0(\in\mathcal M_t(\xi_0))$ $\in_{\xi_0}$-below $\mathsf{K}$ with $\delta_\mathsf{P}\in D^{\gamma}_t$.

\begin{itemize}
\item[$\br$] Suppose $\mathsf{K}\in\pi_1(\mathcal M_p(\gamma))\setminus\mathcal M_{q}(\gamma)$. The conclusion in this case follows by the definition of $\tau^{\gamma}_{\delta_\mathsf{K}}$.

\item[$\br$] Else, $\mathsf{K}\in\pi_1(\mathcal M_q(\gamma))$.
If $\delta_{\mathsf P}\in B^{\gamma}_t$, then $\delta_{\mathsf P}\in B^{\gamma}_q$ and we are done.
Otherwise, $\mathsf P$ is of the right form for $\mathsf K$ in $t$.
In case $\mathsf P$ is of the right form in $q$ we are done by Clause (b4) applied to $q$.
In the other case $\mathsf P$ is not of the right form in $q$.
As we assumed it is of the right form in $t$, it follows that $\gamma<\beta$.
Thus, as $B^{\gamma}_p=B^{\gamma}_q$ by Clause (2) above,
there is no $\mathsf S$ $\xi_0$-above $\mathsf K$ and $\delta_{\mathsf S}\in B^{\gamma}_q$.
This is a contradiction.

As all $S\in\mathcal M_t(\xi_0)$ which are $\in_{\xi_0}$-below $\mathsf{K}$ are in $\mathcal M_q(\xi_0)$, the conclusion holds by the fact $q$ is a condition.

\end{itemize}
\end{proof}

\begin{lemma}
Let $\alpha\in E$. For $p\in\mathbb O_{\alpha}$, suppose $\mathsf{M}\in\pi_1(\mathcal M_p)\cap\mathcal U_{st}$ and let $q\in\mathbb O_{\alpha}\cap \mathsf{M}$ extending $p\rest \mathsf{M}$ satisfying Clauses (1)--(4) above.
Suppose $\gamma\in\dom(F_t)$ is such that $t\res\gamma$ is a legitimate condition. Then $t\res(\gamma+1)$ satisfies Clause (b1) for $t$.
\end{lemma}
\begin{proof}
There are two main cases to consider:

\underline{Case 1:} Suppose $\gamma\in\dom(F_p)\setminus\dom(F_q)$. Given $\mathsf{K}\in\pi_1(\mathcal M_t(\gamma))$ and $\mathsf{P}\in\pi_1(\mathcal M_t(\gamma))$ which is $\in_{\gamma}$ below $\mathsf{K}$.

\begin{itemize}
\item[$\br$] Suppose $\mathsf{K}\in\pi_1(\mathcal M_q(\gamma))$. The conclusion in this case follows by the definition of $\tau^{\gamma}_{\delta_\mathsf{K}}$.

\item[$\br$] Else, $\mathsf{K}\in\pi_1(\mathcal M_p(\gamma))$.

$\br\br$ If $\mathsf{P}\in\mathcal M_{p}(\xi_0)$,
then the conclusion holds by the fact $p$ is a condition.

$\br\br$ Else $\mathsf{P}\in\mathcal M_q(\xi_0)\setminus\mathcal M_p(\xi_0)$. Hence, $\delta_\mathsf{P}>\epsilon^{\gamma}$.
So, by the choice of $q$, $t\res\gamma\forces\delta_p\notin\acc^+(\nacc(\dot{C}^{\gamma}_{\delta_\mathsf{K}}))$.
Thus, the conclusion is immediate.
\end{itemize}

\underline{Case 2:} Otherwise, $\gamma\in\dom(F_q)$. Given $\mathsf{K}\in\pi_1(\mathcal M_t(\gamma))$ and $\mathsf{P}\in\pi_1(\mathcal M_t(\gamma))$ which is $\in_{\gamma}$ below $\mathsf{K}$.
\begin{itemize}
\item[$\br$] Suppose $\mathsf{K}\in\pi_1(\mathcal M_p(\gamma))\setminus\mathcal M_q(\gamma)$. Then, the conclusion in this case follows by the definition of $\tau^{\gamma}_{\delta_\mathsf{K}}$.

\item[$\br$] Else, $\mathsf{K}\in\pi_1(\mathcal M_q(\gamma))$. As all $S\in\pi_1(\mathcal M_t(\xi_0))$ which are $\in_{\gamma}$-below $\mathsf{K}$ are in $\mathcal M_q(\gamma)$, the conclusion holds by the fact $q$ satisfies Clause (b1). \qedhere

\end{itemize}
\end{proof}

Here is our second properness theorem.
\begin{lemma}
Let $\alpha\in E$. Suppose $\mathsf{M}\in\mathcal U_{st}$ with $\eta(\mathsf{M})>\alpha$ and $\alpha\in \mathsf{M}$.
For $r\in \mathsf{M}\cap\mathbb O_{\alpha}$, $r_\mathsf{M}$ is $(\mathsf{M},\mathbb O_{\alpha})$-generic.

\end{lemma}
\begin{proof}
We proceed by induction on $E$.
Note that $\mathbb O_{\min(E)}$ is $\mathcal U_{st}$-strongly proper.
Suppose now the conclusion holds for any ordinal in $E\cap\alpha$.
Suppose that $D\s\mathbb O_{\alpha}$ dense belonging to $\mathsf{M}$.
Let $p\leq r_\mathsf{M}$. Define $\beta:=\max(\dom(F_p)\cap\beta)$.
By the argument above find $q\in D$ witnessing Clauses (1)--(4) above. The lemmata show that $p$ and $q$ are indeed compatible.
Therefore, $r_\mathsf{M}$ is $(\mathsf{M},\mathbb O_{\alpha})$-generic.
\qedhere
\end{proof}

\section{Chain condition} \label{sec5}
Given the work done in the previous sections, the only task left with respect to the analysis of the cardinals structure in our forcing extension is showing the following:

\begin{lemma}
The forcing notion $\mathbb O_{\lambda}$ has the $\lambda$-Knaster property.
\end{lemma}
\begin{proof}
Let $\langle p_i\mid i<\lambda\rangle$ be a set of conditions in $\mathbb O_\lambda$.
We need to find two distinct $i,j<\lambda$ such that $p_i,p_j$ are compatible.
For a given $i<\lambda$ let $a_i:=\bigcup\{a(\mathsf M)\mid \mathsf M\in\mathcal M_{p_i}\}$.
Note that for each $i<\lambda$, $a_i$ is a $<\kappa$ closed subset of $E$.
By applying the $\Delta$-system lemma and counting there exist $A\in[\lambda]^{\lambda}$ such that:
\begin{itemize}
\item[$\br$] $\langle F_{p_i}\mid i\in A\rangle$ is a head-tail-tail $\Delta$ system with root $F$;
\item[$\br$] $\langle a_i\mid i\in A\rangle$ is a head-tail-tail $\Delta$ system with root $a$;
\item[$\br$] for $\gamma=\max\{\max(a),\max(F)\}$ there exists $\mathcal M\in\mathbb M^{\kappa}_{\gamma}$ such that for any $i\in A$, $\mathcal M_{p_i}\rest\gamma=\mathcal M$.
\end{itemize}

Now pick two distinct $i,j\in A$. Define $t:=\langle\mathcal M_t,d_t,F_t\rangle$ as follows:
\begin{itemize}
\item $\mathcal M_r:=\mathcal M_p\cup\mathcal M_q$;
\item $d_r(\mathsf{M}):=\begin{cases} d_p(\mathsf{M}) & \mathsf{M}\in\mathcal M_p;\\
d_q(\mathsf{M}) & \text{otherwise}.
\end{cases} $
\item $\dom(F_r):=\dom(F_p)\cup\dom(F_q)$.
\item $F_r(\gamma):=\begin{cases}
F_p(\gamma),& \gamma\in\dom(F_p);\\
F_q(\gamma),& \gamma\in\dom(F_q).
\end{cases}$.
\end{itemize}

It is straightforward to check that $r\in\mathbb O_{\lambda}$ and $r\leq p,q$.
\end{proof}

\section{Connecting the dots} \label{sec6}

In this final section we show that we have achieved our goal.
\begin{definition}
Suppose $G$ is $V$-generic over $\mathbb O_{\lambda}$ and $\alpha\in E$ is an ordinal of cofinality less than $\kappa$. Let:
\begin{itemize}
\item[(i)] $\mathcal M^{\alpha}_G:=\bigcup_{p\in G}\mathcal M^{\alpha}_p$;
\item[(ii)] $C_{\alpha}(G):=\{\delta_{\mathsf M}\mid\mathsf M\in\mathcal M^{\alpha}_G\}$;
\item[(iii)] $D_{\alpha}(G):=\bigcup_{p\in G} D^{\alpha}_p$.
\end{itemize}

Notice that, $D_{\alpha}(G)\s C_{\alpha}(G)$ and by Clause (b2),
$\Delta_1(D_{\alpha}(G))=\Delta_1(C_{\alpha}(G))$.
\end{definition}

The forcing $\mathbb P^\kappa_\lambda$ adds a tower of $\lambda$ many clubs in $\kappa$.
A verification similar to the one from \cite{MR3821634} shows that our forcing produces such a tower.
The details now follow.
\begin{lemma}
Let $G$ be $V$-generic over $\mathbb O_{\lambda}$ and $\alpha\in E$ of cofinality $<\kappa$.
Then $C_{\alpha}(G)$ is a club in $\kappa$.
Moreover, if $\alpha<\beta$ in $E$ of cofinality $<\kappa$, then $C_{\beta}(G)\setminus C_{\alpha}(G)$ is bounded in $\kappa$.
\end{lemma}
\begin{proof}

We start with the second statement.
Given $\alpha<\beta$ in $E$ of cofinality $<\kappa$, by a density argument we may find $p\in G$ and a Magidor model $\mathsf M\in\mathcal M_G$ which is active at both $\alpha,\beta$.
Therefore, any model $\mathsf N$ which is $\in_{\beta}$ above $\mathsf M\rest\beta$ in $\mathcal M^{\beta}_G$ is active at $\alpha$.
We work in $V$ and prove the statement by induction on $\alpha$.
Let $\dot{\mathcal M}_{\alpha}$ and $\dot{C}_{\alpha}$ be a canonical $\mathbb O_{\lambda}$ names for $\mathcal M^{\alpha}_G$ and $C_{\alpha}(G)$ for $\alpha\in E$.
Fix $\alpha\in E$ of cofinality $<\kappa$ and suppose that for any $\beta\in E\cap\alpha$ the statement is true.
Let $\nu<\kappa$ and suppose that some $p\in\mathbb O_{\lambda}$ forces that
$\nu\in\acc(\dot{C}_{\alpha})$ but $\nu\notin\dot{C}_{\alpha}$.
We may assume, by extending $p$ if necessary, that there is a model $\mathsf M$ such that $p$ forces $\mathsf M$ to be the least in the $\alpha$-chain $\dot{\mathcal M}^{\alpha}$ such that $\nu<\delta_{\mathsf M}$.
Let $\mathsf P$ be the last model on $\mathcal M^{\alpha}_p$ $\in_{\alpha}$-below $\mathsf M$.
Note that such a model exists since $p$ forces $\nu\in\acc(\dot{C}_{\alpha})$
and $\delta_{\mathsf P}<\nu$ as $\nu\notin\dot{C}_{\alpha}$.

We now split into two cases:
\begin{itemize}
\item[$\br$] Suppose $\mathsf M$ is strongly active at $\alpha$. Since $\mathsf P$ is $\mathcal M_p$-free by our assumption, we may assume that $\mathsf P\in\dom(d_p)$.
As $\nu<\delta_{\mathsf M}$ we may find $\delta\in\mathsf M$ such that $\nu\leq\delta<\mathsf M$.
We define a condition $q:=\langle\mathcal M_q,d_q,F_q\rangle$ as follows:
\begin{itemize}
\item[(i)] $\mathcal M_q:=\mathcal M_p$;
\item[(ii)] $F_q=F_p$;
\item[(iii)] $\dom(d_q):=\dom(d_p)$;
\item[(iv)] $d_q(\mathsf S):=\begin{cases} d_p(\mathsf P)\cup\{\delta\}& \mathsf S=\mathsf P;\\
d_p(\mathsf S)&\text{otherwise.}
\end{cases}$
\end{itemize}
Then $q$ forces there is no element of $\dot{C}_{\alpha}$ between $\nu$ and $\delta_{\mathsf P}$ contradicting it forces that $\nu\in\acc(\dot{C}_{\alpha})$.

\item[$\br$] Else $\mathsf M$ is not strongly active at $\alpha$.
Then $\mathsf M$ must be countable.
Denote $\mathcal A:=(\hull(\mathsf M, V_{\alpha})\in,\kappa)$, and let $\alpha^*:=\min((\mathsf M\cap V_{\lambda})\setminus\alpha)$, $\bar{\alpha}:=\sup(\mathsf M\cap\alpha)$.
Note that $\alpha^*\in E$ and $\bar{\alpha}$ is a limit point of $E$ of cofinality $\omega$.
In addition $\mathsf P$ is active at $\bar{\alpha}$.
By the second part of our statement we may assume $p$ forces that $\dot{C}_{\alpha}\setminus\dot{C}_{\bar{\alpha}}\s\delta_{\mathsf P}$, so it also forces that $\nu\in\acc(\dot{C}_{\bar{\alpha}})$.
By the inductive hypothesis $\dot{C}_{\bar{\alpha}}$ is forced to be a club.
So, we may find $q\leq p$ and some $\mathsf N\in\mathcal M^{\bar{\alpha}}_q$
such that $\delta_{\mathsf N}=\nu$.

Now for each $\mathsf Q\in(\mathcal M_q\res\mathsf M)^{\bar{\alpha}}$, find the unique model $\mathsf Q^*\in\mathsf M$ such that $\mathsf Q^*\in\mathcal V^{\mathcal A}_{\alpha^*}$ and $\mathsf Q^*\res\bar{\alpha}=\mathsf Q$.
Denote $\mathcal M^*:=\{\mathsf Q^*\mid\mathsf Q\in(\mathcal M_p\res\mathsf M)^{\bar{\alpha}}\}$.
Working in $\mathcal A$, $\mathcal M^*$ is an $\in_{\alpha^*}$-chain closed under meets that are active at $\alpha^*$.
Let $\mathcal M:=\{\mathsf Q^*\res\alpha\mid \mathsf Q^*\in\mathcal M^*\}$.
Then, $\mathcal M\in_{\alpha}\mathsf M$ and it is an $\in_{\alpha}$-chain closed under meets which are active at $\alpha$.
In addition, $(\mathcal M_q\res\mathsf M)^{\alpha}\s\mathcal M$ and $\Delta(\mathcal M_q(\gamma))=\Delta(\mathcal M(\gamma))$ for all $\gamma\in\dom(F_q)\cap\alpha$.

We define a condition $r:=\langle\mathcal M_r,d_r,F_r\rangle$ as follows:
\begin{itemize}
\item[(i)] $\mathcal M_r$ is the closure under meets of $\mathcal M_p\cup\mathcal M$;
\item[(ii)] $d_r:=d_p$;
\item[(iii)] $F_r:=F_q$.
\end{itemize}

Using \cite[Lemma 3.25]{MohammadpourVelickovic} one may verify that $r$ is indeed a condition extending $q$.
Recall that $\mathsf N\in\mathcal M^{\bar{\alpha}}_q$ and $\delta_{\mathsf N}=\nu$.
Let $\mathsf Q$ be the immediate predecessor of $\mathsf M\rest\bar{\alpha}$ in $\mathcal M^{\bar{\alpha}}_r$.
Then, $\mathsf Q^*\in\mathcal M^*$ and therefore $\mathsf Q^*\rest\alpha\in\mathcal M$.
Denote $R:=\mathsf Q^*\rest\alpha$ and note that $\delta_{\mathsf Q}=\delta_{\mathsf R}$.
But, $r$ forces that $\nu\leq\delta_{\mathsf R}\leq\delta_{\mathsf M}$.
This is a contradiction to the fact that $p$ forces that $\mathsf M$ is the least model in $\dot{\mathcal M}^{\alpha}$ such that $\nu\leq\delta_{\mathsf M}$.
The proof is now complete.\qedhere

\end{itemize}

\end{proof}

\begin{cor} For every $\alpha\in E$ of cofinality less than $\kappa$,
$\Delta_1(D_{\alpha}(G))$ is a club in $S^2_1$. \qed
\end{cor}

\begin{lemma}
$\mathbb O_{\lambda}$ forces the failure of $\mho(S^2_1,\omega_1)$.
\end{lemma}
\begin{proof} Suppose not.
Fix an $\mathbb O_{\lambda}$-nice name $\dot{h}$ for a sequence $\langle (h_\delta,C_\delta)\mid\delta\in S^2_1\rangle$ witnessing $\mho(S^2_1,\omega_1)$.
As $\mathbb O_{\lambda}$ has the $\lambda$-cc. and $\dot{h}$ is a nice name for set of size $\aleph_2$, we may find $\alpha<\omega_3$ large enough such that all the non-trivial condition deciding the value of $\dot{h}$ are contained in $\mathbb O_{\alpha}$

It thus follows that $\dot{h}$ admits a $\mathbb O_{\alpha}$ nice name, say $\sigma$.
Therefore, we may find $\gamma<\omega_3$ above $\alpha$ such that $\phi(\gamma)=\sigma$.
In this case, for $\gamma^+:=\min(E\setminus(\gamma+1))$, $\mathbb O_{\gamma^+}$ introduces the club $\Delta_1(D_{\gamma}(G))$
satisfying that for every $\delta$ in this club and of cofinality $\omega_1$,
there exists a $\tau_{\delta}<\omega_1$ such that, $\tau_{\delta}\notin h_{\delta}[\Delta_1(D_{\gamma}(G))]$.
\end{proof}

\section{Acknowledgments}

This author is supported by the Israel Science Foundation (grant agreement 203/22).
This result is taken from the author's Ph.D. thesis written under the supervision of Assaf Rinot at Bar-Ilan University.
I thank him for presenting me the problem and for many useful suggestions.

\end{document}